\documentclass[11pt]{amsart}

\def\pr{\partial}

\newcommand{\ddt}{\frac{\rm d}{{\rm d} t} }

\if {

}{{\caption}}
} \fi

\def\fh{\hat{f}}

\def\ph{\hat{p}}

\def\uh{\hat{u}}
\def\vh{\hat{v}}
\def\wh{\hat{w}}
\def\xh{\hat{x}}
\def\yh{\hat{y}}

\def\hb{\bar{h}}

\def\pb{\bar{p}}

\def\ub{\bar{u}}
\def\vb{\bar{v}}
\def\wb{\bar{w}}
\def\xb{\bar{x}}
\def\yb{\bar{y}}

\def\xib{\bar{\xi}}

\def\C{\mathcal{C}}

\def\F{\mathcal{F}}
\def\G{\mathcal{G}}
\def\H{\mathcal{H}}

\def\J{\mathcal{J}}

\def\L{\mathcal{L}}

\def\P{\mathcal{P}}
\def\Q{\mathcal{Q}}
\def\R{\mathcal{R}}

\def\U{\mathcal{U}}
\def\V{\mathcal{V}}
\def\W{\mathcal{W}}
\def\X{\mathcal{X}}
\def\Y{\mathcal{Y}}

\def\eps{\varepsilon}

\def\half{\mbox{$\frac{1}{2}$}}

\def\1B{{\bf  1}}

\newcommand{\cR}{\mathbb{R}}

\newcommand\be{\begin{equation}}
\newcommand\ee{\end{equation}}
\newcommand{\benl}{\begin{equation*}}
\newcommand{\eenl}{\end{equation*}}
\newcommand\ba{\begin{array}}
\newcommand\ea{\end{array}}
\newcommand{\bean}{\begin{eqnarray*}}
\newcommand{\eean}{\end{eqnarray*}}

\newcommand{\bs}{\bigskip}

\def\ds{\displaystyle}

\def\la{\langle}
\def\ra{\rangle}

\newcommand{\tras}{^\top}
\newcommand{\intT}{\int_0^T }

\newcommand{\dtt}{\mathrm{d}t}
\newcommand{\mr}{\mathrm}

\usepackage{amsmath,amssymb,latexsym,amsthm,mathdots,yhmath}

\newtheorem{theorem}{Theorem}[section]
\newtheorem{lemma}[theorem]{Lemma}
\newtheorem{proposition}[theorem]{Proposition}
\newtheorem{corollary}[theorem]{Corollary}

\theoremstyle{remark}
{
    \newtheorem{definition}[theorem]{Definition}
    
    \newtheorem{remark}[theorem]{Remark}

\newtheorem{assumption}[theorem]{Assumption}
}

\DeclareMathAlphabet{\mathpzc}{OT1}{pzc}{m}{it}

\newcommand{\gr}{>}
\newcommand{\mi}{<}
\newcommand{\shoot}{\mathcal{S}}
\newcommand{\lam}{[\lambda]}
\newcommand{\lamh}{[\hat\lambda]}
\newcommand{\lamLQ}{[\lambda^{LQ}]}


\begin{document}

\title[Singular solutions in optimal control]{
Singular solutions in optimal control: \\ second order conditions  and  
\\ a shooting algorithm$^{1}$}

\author[M.S. Aronna]{M. Soledad Aronna}
\address{M.S. Aronna\\ Fellowship within the ITN Marie Curie Network SADCO at Universit\`a degli Studi di Padova, Italy}
\email{aronna@math.unipd.it}


\maketitle

\footnotetext[1]{This work is supported by the European Union under the 7th Framework Programme «FP7-PEOPLE-2010-ITN»  Grant agreement number 264735-SADCO.}

\begin{abstract} 
In this article we study optimal control problems for systems that are affine in one part of the control variable. Finitely many equality and inequality constraints on the initial and final values of the state are considered.
We investigate singular optimal solutions for this class of problems. First, we obtain second order necessary and sufficient conditions for weak optimality. 
Afterwards, we propose a shooting algorithm and show that the sufficient condition above-mentioned is also sufficient for the local quadratic convergence of the algorithm.
\end{abstract}

\section{Introduction}\label{Introduction}

The purpose of this paper is to investigate optimal control problems governed by systems of ordinary differential equations that we write as 
\benl
\dot{x}_t=\sum_{i=0}^m v_{i,t} f_i(x_t,u_t),\quad  {\rm  a.e.}\ {\rm  on}\ [0,T].
\eenl
Here $x:[0,T]\to \cR^n$ is the state variable,  $v_0\equiv 1$ is a constant, $v_i:[0,T]\to \cR$ are the affine controls for $i=1,\dots m,$  $u:[0,T]\to \cR^l$ is the nonlinear control and $f_i:\cR^{n+l}\to \cR^n$ is a vector field for each $i=0,\dots m.$ 
Note that this kind of system includes both the totally affine case when $l=0$ and the nonlinear case whenever $m=0.$

Many models that enter into this framework can be found in practice and, in particular, in the existing literature. Among these we can mention:  the Goddard's problem  \cite{Goddard} in three dimensions analyzed in Bonnans et al. \cite{BLMT09}, several models concerning the motion of a rocket treated in Lawden \cite{Law63}, Bell and Jacobson \cite{BelJac}, Goh \cite{GohThesis,Goh08}, Oberle \cite{Obe77}, Azimov \cite{Azi05} and Hull \cite{Hul11}; an hydrothermal electricity production problem studied in Bortolossi et al. \cite{BPT02} and Aronna et al. \cite{ABL08}, the problem of atmospheric flight considered by Oberle in \cite{Obe90}, and an optimal production process in Cho et al. \cite{ChoAbaPar93} and Maurer at al. \cite{MauKimVos05}. All the systems investigated in the cited articles are partially affine in the sense that they have at least one affine and one nonlinear control, i.e. $m,l>0.$

The subject of second order optimality conditions for these partially affine problems have been studied by Goh in \cite{Goh66a,Goh67,GohThesis,Goh08}, Dmitruk in \cite{Dmi11}, Dmitruk and Shishov in \cite{DmiShi10}, Bernstein and Zeidan \cite{BerZei90},  and Maurer and Osmolovskii \cite{MauOsm09}.
The first works were by Goh, who introduced a change of variables in \cite{Goh66a} and used it to obtain optimality conditions in \cite{Goh66a,Goh66,GohThesis}, always assuming uniqueness of the multiplier. The necessary conditions we present imply those by Goh \cite{Goh66}, when there is only one multiplier.
Recently, Dmitruk and Shishov \cite{DmiShi10} analysed the  quadratic functional associated with the second variation of the Lagrangian function and provided a set of necessary conditions for the nonnegativity of this quadratic functional. Their results are consequence of a second order necessary condition that we present.
In  \cite{Dmi11}, Dmitruk proposed, without proof, necessary and sufficient conditions for a problem having a particular structure: the affine control variable applies to a term depending only on the state variable, i.e. the affine and nonlinear controls are  `uncoupled'. This hypothesis is not used in our work.
The conditions established here coincide with those suggested in Dmitruk \cite{Dmi11}, when the latter are applicable.
In \cite{BerZei90}, Bernstein and Zeidan derived a Riccati equation for the singular linear-quadratic regulator, which is a modification of the classical linear-quadratic regulator where only some components of the control enter quadratically in the cost function.
All of these articles use Goh's Transformation to derive their conditions; we use this transformation as well.
On the other hand, in \cite{MauOsm09}, Maurer and Osmolovskii gave a sufficient condition for a class of problems having one affine control 
subject to bounds and such that it is bang-bang at the optimal solution. This structure is not studied here since no control constraints are considered, i.e. our optimal control is suppose to be totally singular.

Regarding second order optimality conditions, we provide a pair of necessary and sufficient conditions for weak optimality of totally singular solutions. These conditions are `no gap' in the sense that  the sufficient condition is obtained from the necessary one by strengthening an inequality. We do not assume uniqueness of multiplier. When our second order conditions are applied to the particular cases having either $m=0$ or $l=0,$ they give already existing results as we point out along the article.

Among the applications of the shooting method to the numerical solution of partially affine problems we can mention the articles Oberle \cite{OberleThesis,Obe90} and Oberle-Taubert \cite{ObeTau97}. In these articles the authors use a generalization of the algorithm that Maurer \cite{Mau76} suggested for totally affine systems.
These works present interesting implementations of a shooting-like method to solve partially affine control problems having bang-singular or bang-bang solutions and, in some cases, running-state constraints are considered. No result on convergence is given in these articles. 

In this paper we propose a shooting algorithm which can be also used to solve problems with bounds on the controls. Our algorithm is an extension of the method for totally affine problems presented in Aronna et al. \cite{ABM11}.
We give a theoretical support to this method by showing that the second order sufficient condition above-mentioned ensures the local quadratic convergence of the algorithm.

The article is organised as follows. In Section \ref{SectionPb} we present the problem, the basic definitions and first order optimality conditions.
In Section \ref{SectionSOC} we give the tools for second order analysis and establish a second order necessary condition. We introduce Goh's Transformation in Section \ref{GohT}.
In Section \ref{SectionNC} we show a new second order necessary condition, and in Section \ref{SectionSC} we give a sufficient one.
A shooting algorithm is proposed in Section \ref{SectionShoot}, and in Section \ref{SectionWP} we prove that the sufficient condition above-mentioned guarantees the local quadratic convergence of the algorithm.

\vspace{8pt}

\noindent\textbf{Notations.}
We denote by $h_t$ the value of function $h$ at time $t$ if $h$ is a function that depends only on $t,$ and by $h_{i,t},$ the $i$th. component of $h$ evaluated at $t.$
Partial derivatives of a function $h$ of $(t,x)$ are referred as $D_th$ or $\dot{h}$ for the derivative in time, and as $D_xh$ or $h_x$ for the differentiations with respect to space variables. The same convention  is extended to higher order derivatives.
By $\cR^k$ we denote the $k-$dimensional real space, i.e. the space of column real vectors of dimension $k;$ and by  $\cR^{k,*}$ its corresponding dual space, which consists of $k-$dimensional row real vectors.
By $L^p(0,T;\cR^k)$ we mean the Lebesgue space with domain equal to the interval $[0,T]\subset \cR$ and with values in $\cR^k.$ The notation $W^{q,s}(0,T;\cR^k)$ refers to the Sobolev spaces (see Adams \cite{Ada75} for further details on Sobolev spaces).

\section{Statement of the problem and assumptions}\label{SectionPb}

\subsection{Statement of the problem.} 
We study the optimal control
problem (P) given by
\begin{align}
&\label{cost} J:=\varphi_0(x_0,x_T)\rightarrow
\min,\\
&\label{stateeq}\dot{x}_t=F(x_t,u_t,v_t),\quad {\rm  a.e.}\ {\rm  on}\ [0,T],\\
& \label{finaleq} \eta_j(x_0,x_T)=0,\quad
\mathrm{for}\
j=1\hdots,d_{\eta},\\
&\label{finalineq} \varphi_i(x_0,x_T)\leq 0,\quad
\mathrm{for}\
i=1,\hdots,d_{\varphi},
\end{align}
where the function $F:\cR^{n+l+m}\to\cR^n$ can be written as
\benl
F(x,u,v):=\sum_{i=0}^m v_{i} f_i(x,u).
\eenl
Here  $f_i:\cR^{n+l}\rightarrow
\cR^n$ for $i=0,\hdots,m,$ 
$\varphi_i:\cR^{2n}\rightarrow \cR$ for
$i=0,\hdots,d_{\varphi},$
$\eta_j:\cR^{2n}\rightarrow \cR$ for
$j=1,\hdots,d_{\eta}$ and we put $v_0\equiv 1,$ which is not a variable. 
The \textit{nonlinear control} $u$ belongs to $\U:=L^{\infty}(0,T;\cR^l),$ while by 
$\V:=L^{\infty}(0,T;\mathbb{R}^m)$ we denote the space of {\it affine controls} $v,$ and 
$\mathcal{X}:=W^{1,\infty}(0,T;\mathbb{R}^n)$ refers to the state space.
 When needed, we write $w=(x,u,v)$ for a point in
$\mathcal{W}:=\mathcal{X}\times \U \times \V.$ 
The hypothesis below is considered along all the article.

\begin{assumption} 
\label{regular}
All data functions have Lipschitz-continuous second derivatives.
\end{assumption}

A \textit{trajectory} is an element $w\in\mathcal{W}$
that satisfies the state equation \eqref{stateeq}.
If in addition, constraints \eqref{finaleq} and
\eqref{finalineq}
hold, we say that $w$ is a \textit{feasible trajectory}
of problem (P). 

\begin{definition}
A feasible trajectory $\wh=(\xh,\uh,\vh)\in\W$ is
 a {\em weak minimum} of (P) if
there exists $\eps>0$ such that the cost function
attains at $\wh$ its minimum in the set of feasible
trajectories $w=(x,u,v)$ satisfying
\benl
\|x-{\xh}\|_{\infty}<\varepsilon,\quad
\|u-\uh\|_{\infty}<\varepsilon,\quad 
\|v-\vh\|_{\infty}\mi \eps.
\eenl
\end{definition}

In the sequel, we study a nominal feasible
trajectory $\wh=(\xh,\uh,\vh)\in \W.$
An element $\delta w\in \W$ is termed \textit{feasible variation for $\wh,$} if $\wh+\delta
w$ is feasible for (P). 
Let $\lambda=(\alpha,\beta,p)$ be in the space
$\cR^{d_{\varphi}+1,*}\times \cR^{d_{\eta},*}\times
W^{1,\infty}(0,T;\cR^{n,*}).$
Define the \textit{pre-Hamiltonian} function
\benl
H[\lambda](x,u,v,t):=p_t\sum_{i=0}^m v_i
f_i(x,u),
\eenl
the \textit{endpoint Lagrangian} function $\ell:\cR^{2n}\rightarrow \cR$ by
\benl
\ell[\lambda](q):=\sum_{i=0}^{d_{\varphi}}
\alpha_i\varphi_i(q)+\sum_{j=1}^{d_{\eta}}\beta_j
\eta_j(q),
\eenl
and the \textit{Lagrangian} function 
\be\label{lagrangian}
\L[\lambda](w):= \ell[\lambda](x_0,x_T) 
+ \int_0^{T}
p_t\left(\sum_{i=0}^{m}v_{i,t}f_i(x_t,u_t)-\dot
x_t\right)\dtt.
\ee
We assume, in sake of simplicity, that whenever some argument of $f_i,$ $H,$ $\ell,$ $\L$ or their derivatives is omitted,  they are evaluated at $\wh.$ Without
loss of generality we suppose that
\be
\varphi_i(\xh_0,\xh_T)=0,\ \mr{for}\ \mr{all}\ 
i=1,\hdots,d_{\varphi}.
\ee


\subsection{Lagrange multipliers}
We introduce here the concept of multiplier. The second order conditions that we prove in this article are expressed in terms of the second variation of the Lagrangian function in \eqref{lagrangian} and the {set of Lagrange multipliers} associated with $\wh$ that we define below.

\begin{definition}
\label{DefMul}
An element $\lambda=(\alpha,\beta,p)\in
\cR^{d_{\varphi}+1,*}\times \cR^{d_{\eta},*}\times
W^{1,\infty}(0,T;\cR^{n,*})$ is a {\em Lagrange multiplier}
associated with $\wh$ if it satisfies the following conditions.
\begin{align}
\label{nontriv}&|\alpha|+|\beta|=1,\\
&\label{alphapos}\alpha=(\alpha_0,\alpha_1,\hdots,\alpha_{d_{\varphi}})\geq0,
\end{align}
the function $p$ is solution of the \textit{costate
equation} 
\be
\label{costateeq}
-\dot{p}_t=H_x[\lambda](\xh_t,\uh_t,\vh_t,t),
\ee
and it satisfies the \textit{transversality conditions}
\be
\label{transvcond}
\begin{split}
p_0&=-D_{x_0}\ell\lam(\xh_0,\xh_T),\\
p_T&=D_{x_T}\ell\lam(\xh_0,\xh_T),
\end{split}
\ee
and the \textit{stationarity conditions} 
\be
\label{stationarity}
\left\{
\ba{l}
\vspace{3pt} \ds  H_u\lam(\xh(t),\uh(t),\vh(t),t)=0,\\
H_v\lam(\xh(t),\uh(t),\vh(t),t)=0,
\ea
\right.
\quad {\rm a.e.}\ {\rm on}\ [0,T],
\ee
hold true. Denote by $\Lambda$ the {\em set of Lagrange multipliers} associated with $\wh.$
\end{definition}

Recall the following well-known result. 
\begin{theorem}
\label{LambdaCompact}
If $\wh$ is a weak minimum for (P), then the set $\Lambda$ is non empty and compact.
\end{theorem}

\begin{proof}
Regarding the existence of a Lagrange multiplier the reader is referred to \cite{AleTikFom79,KurZow} or Milyutin-Osmolovskii \cite[Thm. 2.1]{MilOsm98}. 
In order to prove the compactness, observe that
$p$ may be expressed as a linear continuous
mapping of $(\alpha,\beta).$ Thus, since the
normalization \eqref{nontriv} holds, $\Lambda$ is a 
finite-dimensional compact set.
\end{proof}

In view of the previous result, note that  $\Lambda$ can be
identified with a compact subset of $\cR^s,$ where
$s:=d_{\varphi}+d_{\eta}+1.$

The main results of this article are stated on a restricted subset  of multipliers for which the matrix $D^2_{(u,v)^2} H\lam (\wh,t)$ is singular and, consequently, the pairs $(\wh,\lambda)$ result  to be  singular extremals. We point out this fact in more detail in Remark \ref{RemSing}.

Given $(\xb_0,\ub,\vb)\in \cR^n\times \U\times \V,$ consider 
the \textit{linearized state equation}
\begin{align} 
\label{lineareq} 
\dot{\xb}_t &= F_{x,t}\,\xb_t + F_{u,t}\,\ub_t + F_{v,t}\,\vb_t,\quad {\rm a.e.}\  {\rm on}\ [0,T],\\
\label{lineareq0}
\xb(0) &= \xb_0.
\end{align}
The solution $\xb$ of \eqref{lineareq}-\eqref{lineareq0} is called
\textit{linearized state variable.}

For the interest of the reader, we explicit the expression of the matrices involved in \eqref{lineareq}. For each $t\in [0,T],$ the quantity $F_{x,t}$ is an $n\times n-$matrix given by $\sum_{i=0}^m \vh_{i,t} \frac{\pr f_i}{\pr x}(\xh_t,\uh_t),$ $F_{u,t}$ is $n\times l$ and is equal to $\sum_{i=0}^m \vh_{i,t} \frac{\pr f_i}{\pr u}(\xh_t,\uh_t)$ and, finally, $F_{v,t}$ is an $n\times m-$matrix whose $i$th. column is given by $f_i(\xh_t,\uh_t),$ for $i=1,\dots,m.$


\subsection{Critical cones}\label{ParCritical}
We define now the sets of critical directions associated with $\wh,$ both in the $L^{\infty}-$ and the $L^2-$ norms. Even if we are working with control variables in $L^{\infty}$ and hence the control perturbations are naturally taken in $L^{\infty},$ the second order analysis involves quadratic mappings and it is useful to extend them continuously to $L^2.$ 

Set
$\X_2:=W^{1,2}(0,T;\cR^n),$ $\U_2:=L^2(0,T;\cR^l)$ and
$\V_2:=L^2(0,T;\cR^m).$ Put $\W_2:=\X_2\times
\U_2\times \V_2$ for the corresponding product
space. 
Given $\wb\in\W_2$ satisfying \eqref{lineareq}-\eqref{lineareq0}, consider
the \textit{linearization of the endpoint constraints and cost function,}
\begin{gather}
\label{linearconseq}
D\eta_j(\xh_0,\xh_T)(\xb_0,\xb_T)=0,\quad {\rm for}\  j=1,\hdots,d_{\eta},
\\
\label{linearconsineq}
 D\varphi_i(\xh_0,\xh_T)(\xb_0,\xb_T)\leq 0,\quad {\rm for}\ i=0,\hdots,d_{\varphi}.
\end{gather}
Define the \textit{critical cones} in $\W_2$ and $\W$ by
\begin{gather}
\label{C2}\C_2:=\{\wb\in\W_2:\text{\eqref{lineareq}-\eqref{lineareq0}}\ \text{and}\ \text{\eqref{linearconseq}-\eqref{linearconsineq}}\ \text{hold}\},\\
\label{C} \C:= \C_2 \cap \W.
\end{gather}
\begin{lemma}
\label{conedense}
The critical cone $\C$ is dense on $\C_2.$
\end{lemma}

In order to prove previous lemma, recall the
following technical result (see e.g. Dmitruk
\cite[Lemma 1]{Dmi08} for a proof). 

\begin{lemma}[on density of cones]
\label{lemmadense}
Consider a locally convex topological space $X,$ a
finite-faced cone $Z\subset X,$ and a linear
space $Y$ dense in $X.$ Then the cone $Z\cap Y$ is dense in $Z.$
\end{lemma}


\if{
Let the cone $C$ be given by

\be
C=\{x\in X:(p_i,x)=0,\ \mathrm{for}\ i=
1,\hdots,\mu,\
(q_j,x)\leq 0,\ \mathrm{for}\ j=1,\hdots,\nu\}.
\ee

Let us show first, without lost of generality, that
the
equality constraints can be removed from the
formulation.
It suffices to consider the case where $C$ is given
by only
one equality $(p,x)=0.$

Take any point $x_0\in C$ and a convex
neighborhood$\mathcal{O}(x_0).$ We have to show that there
exists $x$
in $C\cap L\cap \mathcal{O}(x_0).$ Since the set
$(p,x)<0$
is open, its intersection with $\mathcal{O}(x_0)$ is
open
too and obviously nonempty, hence it contains a
point $x_1$
from the set $L,$ because the last one is dense in
$X.$
Similarly, the intersection of the sets $(p,x)<0$
and
$\mathcal{O}(x_0)$ contains a point $x_2\in L.$
Since
$\mathcal{O}(x_0)$ is convex, it contains a point
$x$ such
that $(p,x)=0,$ which belongs to $C$ and to $L\cap
\mathcal{O}(x_0).$

We can then consider only the case where $C$ is
given by a
finite number of inequalities. Suppose first that
there
exists $\xh\in C$ such that $(q_j,\xh)<0$ for all
$j,$
hence $\xh \in \mathrm{int}{C}.$ Take any $x_0\in C$
and
any convex neighborhood $\mathcal{O}(x_0).$ We have
to find
a point $x\in C\cap \mathcal{O}(x_0)\cap L.$ We know
that,
for any positive $\varepsilon,$ the point
$x_{\varepsilon}:=x_0+\varepsilon \xh$ lies in
$\mathrm{int}(C),$ and then there exists a positive
$\varepsilon$ such that this point lies also in
$\mathcal{O}(x_0).$ Thus, the open set
$\mathrm{int}{C}\cap
\mathcal{O}(x_0)$ is nonempty, and then contains a
point
$x$ from the dense set $L.$

Suppose now that the above point
$\xh\in\mathrm{int}{C}$
does not exist, and, without lost of generality,
that
$q_j\neq 0$ for every $j.$ In this case, by the
Dubovitskii-Milyutin theorem, there exist
multipliers
$\alpha_j\geq0,\ j=1,\hdots,\nu,$ not all zero, such
that
Euler-Lagrange equation holds: $\alpha_1
q_1+\hdots+\alpha_{\nu}q_{\nu}=0.$ Suppose, without
lost of
generality, that $\alpha_{\nu}>0.$ Then, for all $x
\in C$
we actually have $(q_{\nu},x)=0,$ not just $\leq0.$
This
means that the cone $C$ can be given by the
constraints
$(q_j,x)\leq 0,\ j=1,\hdots,\nu-1,\ (q_{\nu},x)=0.$
But, as
was already shown, the last equality can be removed,
so the
cone can be given by a smaller number of
inequalities.
Applying induction arguments, we arrive at a
situation when
either all the inequalities are changed into
equalities and
then removed, or the strict inequalities have a
nonempty
intersection. Since both cases are already
considered, the
proof is complete.
}\fi


\begin{proof}
[of Lemma \ref{conedense}] 
Set $X:=\{\wb\in\W_2:\text{\eqref{lineareq}-\eqref{lineareq0}}\, \text{hold}\},$ 
$Y:=\{\wb\in\W:\text{\eqref{lineareq}-\eqref{lineareq0}}\, \text{hold}\},$ $Z:=\C_2$ and apply Lemma \ref{lemmadense}.
\end{proof}


\section{Second order analysis}\label{SectionSOC}
We begin this section by giving an expression of the second order derivative of the Lagrangian function $\L,$ in terms of derivatives of $\ell$ and $H.$ We denote the related quadratic mapping by $\Omega.$
All the second order conditions we present are established in terms of either $\Omega$ or some transformed form of $\Omega.$ The main result of the current section is the necessary condition in Theorem \ref{strengthNC}, which is applied in Section \ref{SectionNC} to get the stronger condition given in  Theorem \ref{NCP2}.

\subsection{Second variation}

Let us consider the quadratic mapping
\be 
\label{Omega}
\begin{split}
\Omega  \lam  (\xb,\ub,\vb):=& \,
\half D^2\ell\lam(\xh_0,\xh_T)(\xb_0,\xb_T)^2  + \intT \left[\half\xb\tras H_{xx}\lam \xb \right. \\
& \left.+\, \ub\tras H_{ux}\lam\xb +
\vb\tras H_{vx}\lam\xb + \half\ub\tras H_{uu}\lam\ub + \vb\tras
H_{vu}\lam\ub\right] \dtt,
\end{split}
\ee
where, whenever the argument is missing, the corresponding function is evaluated on the reference trajectory $\wh.$ 

Note that $H_{xx}\lam = p\sum_{k=0}^m \vh_k \frac{\pr^2 f_k}{\pr x^2}$ is an $n\times n-$matrix, 
$H_{ux}\lam$ is equal to $p\sum_{k=0}^m \vh_k \frac{\pr^2 f_k}{\pr u\pr x}$ and it is $l\times n,$ 
$H_{vx}\lam$ is an $m\times n-$matrix whose $i$th. row is given by $H_{v_ix}\lam= p \frac{\pr f_i}{\pr x}.$ The matrix $H_{uu}\lam$ is equal to $p\sum_{k=0}^m \vh_k \frac{\pr^2 f_k}{\pr u^2}$ and it is $l\times l,$ and the $i$th. row of the $m\times l-$matrix $H_{vu}\lam$ is $H_{v_i u}\lam= p \frac{\pr f_i}{\pr u}.$


The result below gives an expression of the Lagrangian $\L$ around the nominal trajectory $\wh.$ For the sake of simplicity, the time variable is omitted in the statement.

\begin{lemma}[Lagrangian expansion] 
\label{expansionlagrangian}
Let $w=(x,u,v)\in\W$ be a trajectory and set $\delta w=(\delta x,\delta u,\delta v):=w-\wh.$
Then for every multiplier $\lambda\in\Lambda,$ 
\be
\label{expansionLag}
\L\lam(w) 
=\L\lam(\wh) 
+ \Omega\lam(\delta x,\delta u,\delta v)  
+\omega\lam(\delta x,\delta u,\delta v)  +\R(\delta x,\delta u,\delta v),
\ee
where  $\omega$ is a cubic mapping given by
\begin{align*}
\omega\lam&(\delta x,\delta u,\delta v)  := \\
&\intT \left[ H_{vxx}\lam(\delta x, \delta x,\delta v) + 2H_{vux}\lam(\delta x,\delta u,\delta v) 
+ H_{vuu}\lam(\delta u,\delta u,\delta v) \right] \dtt,
\end{align*}
and $\R$ satisfies the estimate
\benl
\R(\delta x,\delta u,\delta v)  = L_\ell |(\delta x_0,\delta x_T)|^3
+L K (1+\|v\|_\infty)\, \|(\delta x, \delta u)\|_{\infty}\|(\delta x,\delta u) \|_2^2.
\eenl
Here $L_\ell$ is a Lipschitz constant for $D^2\ell\lam$ uniformly in $\lambda\in \Lambda,$ $L$ is a Lipschitz constant for $D^2f_i$ for all $i=0,\dots,m,$ and $K:=\ds\sup_{\lambda\in \Lambda} \|p\|_\infty.$
\end{lemma}

\begin{proof}
We shall omit the dependence on $\lambda$ for the sake of
simplicity. 
Let us consider the following second
order Taylor expansions, written in a compact form, 
\begin{eqnarray} 
\label{expell} 
&& \ell (x_0,x_T)
= \ell + {D\ell} (\delta x_0,\delta x_T)
+ \half {D^2\ell} (\delta x_0,\delta x_T)^2
+ L_\ell|(\delta x_0,\delta x_T)|^3, 
\\
\label{expfi}
 && f_i(x_t,u_t) = f_{i,t} + {D f_{i}} (\delta x_t,\delta u_t)
+ \half {D^2 f_{i}} (\delta x_t,\delta u_t)^2
+L |(\delta x_t,\delta u_t) |^3.
\end{eqnarray}
Observe that, in view of the transversality conditions \eqref{transvcond} and the costate equation \eqref{costateeq}, one has
\be
\label{Dell}
\begin{split}
{D\ell}\,(\delta x_0,\delta x_T) 
&= -p_0\, \delta x_0 + p_T \, \delta x_T \\
&= \intT \left[ \dot{p}\, \delta x + p \dot{\delta x} \right] \dtt =
\intT p\,\Big[-\sum_{i=0}^m\vh_i {D_xf_i}\,  \delta x +\dot{\delta x}\Big] \dtt.
\end{split}
\ee
In the definition of $\L$ given in \eqref{lagrangian}, replace $\ell(x_0,x_T)$ and $f_i(x,u)$  by their Taylor
expansions \eqref{expell}-\eqref{expfi} and use the identity \eqref{Dell}. This  yields
\begin{align*}
\L(w) 
=&\,\,\L(\wh) + \ds \intT \big[ H_u\delta u + H_v\delta v \big]\dtt + \Omega(\delta x,\delta
u,\delta v) \\
&+ \ds \intT \big[ H_{vxx}(\delta x, \delta x,\delta v) 
+ 2H_{vux}(\delta x,\delta u,\delta v) 
+ H_{vuu}(\delta u,\delta u,\delta v) \big] \dtt \\
&+ \, L_\ell |(\delta x_0,\delta x_T)|^3
+ L(1+\|v\|_\infty)\,\|(\delta x,\delta u) \|_{\infty} \ds\intT p\,|(\delta x,\delta u) |^2 \dtt.
\end{align*}
Finally, to obtain \eqref{expansionLag}, remove the first order terms by the stationarity conditions \eqref{stationarity}, and use the Cauchy-Schwarz inequality in the last integral.
This completes the proof.
\end{proof}

\begin{remark} From previous lemma one gets the identity
\be 
\Omega\lam (\wb) = \half D^2\L\lam (\wh)\, \wb^2.
\ee
\end{remark}


\subsection{Second order necessary condition}

Recall the second order condition below.

\begin{theorem}[Classical second order necessary
condition]
\label{classicalNC}
 If $\wh$ is a weak minimum of problem (P), then
\be \label{classicalNCeq}
\max_{\lambda\in \Lambda} \Omega \lam (\xb,\ub,\vb)
\geq 0,\ \mr{for}\ \mr{all}\ (\xb,\ub,\vb)\in\C.
\ee
\end{theorem}

A proof of Theorem \ref{classicalNC} can be found in Osmolovskii \cite{OsmThesis}.
Nevertheless, for the sake of completeness, we give a proof here.

We shall write problem (P) in an abstract form and, therefore,  we define for $j=1,\dots,d_\eta$ and $i=0,\dots,d_\varphi,$
\begin{gather*}
\bar{\eta}_j:\cR^n\times \U \times \V \rightarrow \cR,\quad (x_0,u,v)\mapsto \bar{\eta}_j(x_0,u,v):=\eta_j(x_0,x_T),\\
\bar{\varphi}_i:\cR^n\times \U \times \V \rightarrow \cR,\quad (x_0,u,v)\mapsto \bar{\varphi}_i(x_0,u,v):=\varphi_i(x_0,x_T),
\end{gather*}
where $x\in\W$ is the solution of \eqref{stateeq} associated to $(x_0,u,v).$ Hence, (P) can be written as the following problem in the space $\cR^n\times \U \times \V,$
\be
\label{AP}\tag{AP}
\begin{split}
\min\,\,&\bar\varphi_0(x_0,u,v);\\
\,\,\text{s.t.}\,\,
&\bar\eta_j(x_0,u,v)=0,\ \text{for}\ j=1,\dots,d_{\eta},\\
&\bar\varphi_i(x_0,u,v)\leq 0,\ \text{for}\ j=1,\dots,d_{\varphi}.
\end{split}
\ee
Note that if $\wh$ is a weak solution of (P) then $(\xh_0,\uh,\vh)$ is a local solution of (AP). 

\begin{definition} 
\label{DefCQ}
We say that the endpoint equality constraints are {\em qualified} if
\be\label{QC}
D\bar\eta(\xh_0,\uh,\vh)\ \text{is onto from}\ \cR^n\times \U\times\V\ \text{to}\ \cR^{d_{\eta}}.
\ee
When \eqref{QC} does not hold, the constraints are {\em not qualified}.
\end{definition}

The proof of Theorem \ref{classicalNC} is divided in two cases: qualified and not qualified endpoint equality constraints. In the latter case the condition \eqref{classicalNCeq} follows easily and it is shown in Lemma \ref{degNC} below. The proof for the qualified case is done by means of an auxiliary  linear problem and duality arguments.

\begin{lemma}\label{degNC}
If the equality constraints are not qualified then \eqref{classicalNCeq} holds.
\end{lemma}

\begin{proof}
Observe that since $D\bar\eta(\xh_0,\uh,\vh)$ is not onto, there exists $\beta\in\cR^{d_{\eta},*}$ with $|\beta|=1,$ such that $\sum_{j=1}^{d_{\eta}} \beta_j D\bar\eta_j(\xh_0,\uh,\vh)=0$ and, consequently, 
\benl
\sum_{j=1}^{d_{\eta}} \beta_j D\eta_j(\xh_0,\xh_T)=0.
\eenl
 Set $\lambda:=(\alpha,\beta,p)$ with  $\alpha=0$ and $p\equiv 0.$ Then both $\lambda$ and $-\lambda$ are in $\Lambda.$ Observe that \benl
\Omega\lam(\xb,\ub,\vb)=
\half \sum_{j=1}^{d_{\eta}} \beta_j D^2\eta_j(\xh_0,\xh_T)(\xb_0,\xb_T)^2.
\eenl
Thus, either $\Omega\lam(\xb,\ub,\vb)$ or $\Omega[-\lambda](\xb,\ub,\vb)$  is nonnegative. The desired result follows.
\end{proof}

Let us now deal with the qualified case. Take a critical direction $\wb=(\xb,\ub,\vb)\in \C$ and consider the problem in the variables $\tau\in\cR$ and $r=(r_{x_0},r_u,r_v)\in\cR^n\times \U\times \V$ given by
\be\label{QPw}\tag{QP$_{\wb}$}
\begin{split}
\min\,\, &\tau \\
\text{s.t.}\,\,\, & D\bar\eta(\xh_0,\uh,\vh) r +D^2\bar\eta(\xh_0,\uh,\vh)(\xb_0,\ub,\vb)^2 =0,\\
& D\bar\varphi_i(\xh_0,\uh,\vh) r +D^2\bar\varphi_i(\xh_0,\uh,\vh)(\xb_0,\ub,\vb)^2 \leq \tau,\ \ i=0,\dots,d_{\varphi}.
\end{split}
\ee

\begin{proposition}\label{dualpb}
Assume that $\wh$ is a weak solution of \eqref{AP} for which the endpoint equality constraints are qualified. Let $\wb\in \C$ be a critical direction. Then the problem  \eqref{QPw} is feasible and has nonnegative value.
\end{proposition}

Recall first the following notation. Given two functions $k_1:\cR^N \rightarrow \cR^{M}$ and $k_2: \cR^N  \rightarrow \cR^L,$ we say that $k_1$ is a {\it big-O} of $k_2$ around 0 and denote it by
\benl
k_1(x) = \mathcal{O} (k_2(x)),
\eenl
if there exists positive constants $\delta$ and $M$ such that $|k_1(x)| \leq M|k_2(x)|$ for $|x|<\delta.$ It is a {\it small-o} if $M$ goes to 0 as $|x|$ goes to 0. Denote this by
\benl
k_1(x) = o(k_2(x)).
\eenl

\begin{proof}
[of Proposition \ref{dualpb}]
Step I. Let us first show feasibility. Since $D\bar\eta(\xh_0,\uh,\vh)$ is onto, there exists $r\in \cR^n\times \U\times \V$ for which the equality constraint in \eqref{QPw} is satisfied. Set
\be
\label{zeta1}
\tau:=\max_{0\leq i \leq d_{\varphi}}\{D\bar\varphi_i(\xh_0,\uh,\vh) r +D^2\bar\varphi_i(\xh_0,\uh,\vh)(\xb_0,\ub,\vb)^2 \}.
\ee
Then $(\tau,r)$ is feasible for \eqref{QPw}.

Step II. Let us now prove that \eqref{QPw} has nonnegative value. Suppose on the contrary that there is $(\tau,r)\in \cR\times\cR^n\times \U\times \V$ feasible for \eqref{QPw}, with $\tau<0.$ 
We shall look for a family of feasible solutions of \eqref{AP}, that we refer as $\{r(\sigma)\}_{\sigma},$ with the following properties:  it is defined for small positive $\sigma$ and  it satisfies 
\be\label{estrsigma}
r(\sigma)\underset{\sigma\to0}{\longrightarrow} (\xh_0,\uh,\vh)\ \text{in} \  \cR^n\times \U\times \V,\ \text{and}\ 
\bar\varphi_0(r(\sigma) )< \bar\varphi_0(\xh_0,\uh,\vh).
\ee
The existence of such family $\{r(\sigma)\}_{\sigma}$ will contradict the local optimality of $(\xh_0,\uh,\vh).$
Consider hence
\benl
\tilde{r}(\sigma):= (\xh_0,\uh,\vh)+\sigma(\xb_0,\ub,\vb)+\half\sigma^2 r.
\eenl
Let $0\leq i\leq d_{\varphi}$ and observe that 
\be\label{estvarphi}
\begin{split}
\bar\varphi_i(\tilde{r}(\sigma))
=&
\,\bar\varphi_i(\xh_0,\uh,\vh)+\sigma D\bar\varphi_i(\xh_0,\uh,\vh)(\xb_0,\ub,\vb)\\
&+\half\sigma^2\left[ D\bar\varphi_i(\xh_0,\uh,\vh)r+D^2 \bar\varphi_i(\xh_0,\uh,\vh)(\xb_0,\ub,\vb)^2 \right] + o(\sigma^2)\\
\leq&\,\bar\varphi_i(\xh_0,\uh,\vh)+\half\sigma^2\tau+o(\sigma^2),
\end{split}
\ee
where last inequality holds since $(\xb,\ub,\vb)$ is a critical direction and in view of the definition of $\tau$  in \eqref{zeta1}.
Analogously, one has
\benl
\bar\eta(\tilde{r}(\sigma))=o(\sigma^2).
\eenl
Since $D\bar\eta(\xh_0,\uh,\vh)$ is onto, there exists $r(\sigma) \in \cR\times \U\times \V$ such that $\| r(\sigma)-\tilde{r}(\sigma)\|_{\infty}=o(\sigma^2)$ and $\bar\eta(r(\sigma))=0.$
This follows by applying the Implicit Function Theorem to the mapping
\benl
(r,\sigma)\mapsto \bar\eta\left((\xh_0,\uh,\vh)+\sigma(\xb_0,\ub,\vb)+\half\sigma^2 r\right)=\bar\eta(\tilde{r}(\sigma)).
\eenl
On the other hand, by taking $\sigma$ sufficiently small 
in estimate \eqref{estvarphi}, we obtain 
\benl
\bar\varphi_i({r}(\sigma))<\bar\varphi_i(\xh_0,\uh,\vh),
\eenl
since $\tau<0.$ Hence $r(\sigma)$ is feasible for \eqref{AP} and  verifies \eqref{estrsigma}. This contradicts the optimality of $(\xh_0,\uh,\vh).$ We conclude then that all the feasible solutions of \eqref{QPw} have $\tau\geq 0$ and, therefore, the value of \eqref{QPw} is nonnegative.
\end{proof}

We shall now go back to the proof of Theorem \ref{classicalNC}.

\begin{proof}
[of Theorem \ref{classicalNC}]
The not qualified case is covered by Lemma \ref{degNC} above. Hence, for this proof, assume that \eqref{QC} holds.

Given $\wb\in \C, $ note that \eqref{QPw} can be regarded as a linear problem in the variables $(\zeta, r),$ whose associated dual  is given by 
\begin{align}
\label{dualQP1}
&\sum_{i=0}^{d_{\varphi}} \alpha_i D^2\bar\varphi_i(\xh_0,\uh,\vh)(\xb_0,\ub,\vb)^2 + \sum_{j=1}^{d_{\eta}} \beta_j D^2\bar\eta_j(\xh_0,\uh,\vh)(\xb_0,\ub,\vb)^2 \rightarrow \max_{(\alpha,\beta)}\\
\label{dualQP2}
& \sum_{i=0}^{d_{\varphi}} \alpha_i D\bar\varphi_i(\xh_0,\uh,\vh)+\sum_{j=1}^{d_{\eta}} \beta_j D\bar\eta_j(\xh_0,\uh,\vh)=0,\\
\label{dualQP3}
&\sum_{i=0}^{d_{\varphi}} \alpha_i=1,\quad \alpha \geq 0.
\end{align}
The Proposition \ref{dualpb} above and the linear duality result in Bonnans \cite[Theorem 3.43]{BonOC} imply that \eqref{dualQP1}-\eqref{dualQP3} has finite nonnegative value  (the reader is referred to Shapiro \cite{Sha01} and references therein for a general theory on linear duality). Consequently, there exists a feasible solution $(\bar\alpha,\bar\beta)\in \cR^{d_\varphi+d_\eta+1}$ to \eqref{dualQP1}-\eqref{dualQP3}, with associated nonnegative and finite value. Set $(\alpha,\beta) := (\bar\alpha,\bar\beta)/(\sum_{i=0}^{d_{\varphi}} |\bar\alpha_i|+\sum_{j=1}^{d_{\eta}} |\beta_j|),$ whose the denominator is not zero, in view of \eqref{dualQP1}. We have $(\alpha,\beta)\in \cR^{d_\varphi+d_\eta+1}$ verifying \eqref{nontriv}-\eqref{alphapos}, \eqref{dualQP2} and such that
\be
\label{DM2}
\sum_{i=0}^{d_{\varphi}} \alpha_i D^2\bar\varphi_i(\xh_0,\uh,\vh)(\xb_0,\ub,\vb)^2+\sum_{j=1}^{d_{\eta}} \beta_j D^2\bar\eta_j(\xh_0,\uh,\vh)(\xb_0,\ub,\vb)^2 \geq0.
\ee
\if{
implies that there cannot exist $(\tau,r)\in \cR\times \cR^n\times \U\times \V$ such that
\benl
\left\{
\begin{split}
&D\bar\eta(\xh_0,\uh,\vh) r +D^2\bar\eta(\xh_0,\uh,\vh)(\xb_0,\ub,\vb)^2 =0,\\
&D\bar\varphi_i(\xh_0,\uh,\vh) r +D^2\bar\varphi_i(\xh_0,\uh,\vh)(\xb_0,\ub,\vb)^2 \leq \tau,\ \text{for}\ i=0,\dots,d_{\varphi},\\
&\tau <0.
\end{split}
\right.
\eenl
Therefore, the Dubovitskii-Milyutin Theorem \cite{DubMil} guarantees the existence of $(\alpha,\beta)\in\cR^{1+d_{\varphi}+d_{\eta}}$ such that $(\alpha,\beta)\neq0,$ $\alpha\geq 0$ and 
\begin{gather}
\label{DM1}
\ds\sum_{i=0}^{d_{\varphi}} \alpha_i D\bar\varphi_i(\xh_0,\uh,\vh) + \sum_{j=1}^{d_{\eta}} \beta_j D\bar\eta_j(\xh_0,\uh,\vh) =0,\\
\label{DM2}
\ds\sum_{i=0}^{d_{\varphi}}\alpha_i D^2\bar\varphi_i(\xh_0,\uh,\vh)(\xb_0,\ub,\vb)^2 + \sum_{j=1}^{d_{\eta}} \beta_j D^2\bar\eta_j(\xh_0,\uh,\vh)(\xb_0,\ub,\vb)^2 \geq 0
\end{gather}
}\fi
For this $(\alpha,\beta),$ let $p$ be the solution of \eqref{costateeq}
with final condition
\be
\label{DM5}
p_T=\sum_{i=0}^{d_{\varphi}} \alpha_i D_{x_T}\varphi_i(\xh_0,\xh_T) + \sum_{j=1}^{d_{\eta}} \beta_j D_{x_T}\eta_j(\xh_0,\xh_T).\ee
We shall prove that $\lambda := (\alpha,\beta,p)$ is in $\Lambda,$ i.e. that also the first line in \eqref{transvcond} and the stationarity conditions \eqref{stationarity} hold. Let $(\tilde{x},\tilde{u},\tilde{v})\in\W$ be solution of the linearized equation \eqref{lineareq}. In view of \eqref{dualQP2},
\be
\label{DM4}
\sum_{i=0}^{d_{\varphi}} \alpha_i D\bar\varphi_i(\xh_0,\uh,\vh)(\tilde{x}_0,\tilde{u},\tilde{v}) + \sum_{j=1}^{d_{\eta}} \beta_j D\bar\eta_j(\xh_0,\uh,\vh)(\tilde{x}_0,\tilde{u},\tilde{v}) =0,
\ee
Hence, rewriting in terms of the endpoint Lagrangian $\ell$ and using \eqref{DM5}-\eqref{DM4}, one has
\benl
\label{DM3}
0 = D\ell \lam (\xh_0,\xh_T)(\tilde{x}_0,\tilde{x}_T) = D_{x_0}\ell \lam (\xh_0,\xh_T)\tilde{x}_0 + p_T \tilde{x}_T \pm p_0 \tilde{x}_0.
\eenl
By regrouping terms in the previous equation, and due to the formula of differentiation of the product, we get
\be
\label{DM6}
\begin{split}
0&= \big( D_{x_0}\ell\lam(\xh_0,\xh_T)+p_0\big) \tilde{x}_0 + \intT (\dot{p}\tilde{x} +p \dot{\tilde{x}} ) \dtt \\
&= \big( D_{x_0}\ell\lam(\xh_0,\xh_T)+p_0\big) \tilde{x}_0 + \intT \big( H_u\lam \tilde{u} + H_v\lam \tilde{v} \big) \dtt,
\end{split}
\ee
where we used  \eqref{costateeq} and \eqref{lineareq} in the last equality. Since \eqref{DM6} holds for all $(\tilde{x}_0,\tilde{u},\tilde{v}) \in \cR^n \times \U\times \V,$ the first line in \eqref{transvcond} and the stationarity conditions in \eqref{stationarity} are necessarily verified.
Thus, $\lambda$ is an element of $\Lambda.$ On the other hand, simple computations yield that  \eqref{DM2} is equivalent to
\benl
\Omega\lam (\xb,\ub,\vb)\geq0,
\eenl
and, therefore, the result follows.
\end{proof}

We can obtain the following extension of \eqref{classicalNCeq} to the cone $\C_2.$

\begin{theorem}
\label{NCC2}
 If $\wh$ is a weak minimum of problem (P), then
\be 
\label{classicalNCeqC2}
\max_{\lambda\in \Lambda} \Omega \lam (\xb,\ub,\vb)
\geq 0,\quad \text{for all}\ (\xb,\ub,\vb)\in \C_2.
\ee
\end{theorem}

\begin{proof}
Observe first that $\Omega\lam$ can be extended to the space $\W_2$ since all the coefficients are essentially bounded. The result follows by  the density property of Lemma \ref{conedense} and  the compactness of $\Lambda$ proved in Theorem \ref{LambdaCompact}. 
\end{proof}


\subsection{Strengthened second order necessary condition}

In the sequel we aim to strengthen previous necessary condition by proving that the maximum in \eqref{classicalNCeqC2} remains nonnegative when taken in a smaller set of multipliers.

We shall first give a description of the subset of Lagrange multipliers we work with.
Set
\be\label{H2}
\H_2:=\{(\xb,\ub,\vb)\in\W_2:\eqref{lineareq}\ \text{holds}\},
\ee
and consider the subset of $\Lambda$ given by
\benl 
\Lambda^{\#}:=\{ \lambda\in \Lambda: \Omega\lam\
\text{is weakly-l.s.c. on }\H_2\},
\eenl
where `l.s.c.' means lower semicontinuous.
We have the following lemma that gives a characterization of $\Lambda^{\#},$ and the Theorem \ref{strengthNC} afterwards which is a new  necessary optimality condition.

\begin{lemma}
\label{Lambdawlsc}
\be 
\Lambda^{\#}=\{ \lambda\in \Lambda: H_{uu}
\lam\succeq0\ {\rm and}\ H_{vu}\lam = 0,\,\, {\rm a.e.}\, {\rm on}\,[0,T]\}.
\ee
\end{lemma}

\begin{theorem}[Strengthened second order necessary condition]
\label{strengthNC}
 If $\wh$ is a weak minimum of problem (P), then
\be \label{strengthNCeq}
\max_{\lambda\in \Lambda^{\#}} \Omega \lam
(\xb,\ub,\vb) \geq 0,\quad \text{for all}\ (\xb,\ub,\vb)\in \C_2.
\ee
\end{theorem}

\begin{remark}
\label{RemSing}
From now on we restrict the set of multipliers to $\Lambda^\#$ or some subset of it and, therefore, $H_{uv}\lam\equiv 0$ along $\wh.$ Consequently, 
\benl
D^2_{(u,v)^2}H\lam (\wh_t,t)\ \mr{is}\ \mr{a}\ \mr{singular}\ \mr{matrix}\ \mr{a.e.}\ \mr{on}\ [0,T].
\eenl
The latter assertion together with the stationarity \eqref{stationarity} imply that $(\wh,\lambda)$ is a {\it singular extremal} (as defined in Bell-Jacobson \cite{BelJac} and Bryson-Ho \cite{BryHo}).
\end{remark}

In order to prove  Lemma \ref{Lambdawlsc} note that $\Omega\lam$ can be written as the sum of two terms: the first one being a weakly-continuous function on the space $\H_2$ given by
\be\label{Omega1}
(\xb,\ub,\vb)\mapsto\half D^2\ell\lam(\xb_0,\xb_T)^2  + \intT \big[\half\xb\tras H_{xx}\lam \xb + \ub\tras H_{ux}\lam\xb +
\vb\tras H_{vx}\lam\xb\big]\dtt,
\ee
 and the second one being the quadratic operator
\be\label{Omega2}
(\ub,\vb)\mapsto 
\intT \big[ \half\ub\tras H_{uu} \lam\ub + \vb\tras
H_{vu}\lam\ub \big] \dtt.
\ee
The weak-continuity of the mapping in \eqref{Omega1} follows easily. On the other hand, in view of Hestenes \cite[Theorem 3.2]{Hes51} the following characterization holds.

\begin{lemma}
\label{Lemmawlsc}
The mapping in \eqref{Omega2} is weakly-lower semicontinuous on $\U\times \V$ if and only if the matrix
\be 
\label{Huvuv}
D^2_{(u,v)^2}H\lam=\begin{pmatrix}
H_{uu}\lam & H_{vu}\lam\tras \\
 H_{vu}\lam & 0\\
\end{pmatrix},
\ee
is positive semidefinite a.e. on $[0,T].$
\end{lemma}

\begin{remark}
\label{remarkLC}
The fact that the matrix in \eqref{Huvuv} is positive semidefinite is known as the {\em Legendre-Clebsch necessary optimality condition} for the extremal $(\wh,\lambda)$ (see e.g. \cite{BryHo,AgrSac} or Corollary \ref{NCunique} below). 
\end{remark}

\begin{proof}
[of Lemma \ref{Lambdawlsc}]
It follows from the decomposition given by \eqref{Omega1}-\eqref{Omega2} and previous Lemma \ref{Lemmawlsc}.
\end{proof}

Finally, to achieve Theorem \ref{strengthNC}, recall the following result on quadratic forms.

\begin{lemma}
\label{quadform}
\cite[Theorem 5]{Dmi84}
Given a Hilbert space $H,$ and
$a_1,a_2,\hdots,a_p$ in $H,$ set
\be
K:=\{x\in H:\la a_i,x \ra\leq 0,\ \mr{for}\
i=1,\hdots,p\}.
\ee
Let $M$ be a convex and compact subset of $\cR^s,$
and let
$\{Q^{\psi}:\psi\in M\}$ be a family of continuous
quadratic forms over $H,$ the mapping $\psi
\rightarrow
Q^{\psi}$ being affine. Set 
$M^{\#}:=\{ \psi \in M:\ Q^{\psi}\ \text{is
weakly-l.s.c.}\text{ on } H\}$ and assume that
\be
\max_{\psi\in M} Q^{\psi}(x)\geq 0,\ \text{for all}\ x\in K.
\ee
Then
\be
\max_{\psi\in
M^{\#}} Q^{\psi}(x)\geq 0,\ \text{for all}\ x\in K.
\ee
\end{lemma}

\begin{proof}
[of Theorem \ref{strengthNC}]
It is a consequence of Theorem \ref{NCC2}, Lemmas \ref{Lambdawlsc} and \ref{quadform}. 
\end{proof}

We finish this section with the following Corollary.

\begin{corollary}[Legendre-Clebsch condition]
\label{NCunique}
If $\wh$ is a weak minimum for (P) with a unique associated multiplier $\hat\lambda,$ then $(\wh,\hat\lambda)$ satisfies the Legendre-Clebsch condition. In order words, the matrix in \eqref{Huvuv} is positive semidefinite and, consequently,
\be
\label{R0pos}
 H_{uu} \lamh\succeq0\ {\rm and}\ H_{vu}\lamh \equiv 0.
\ee
\end{corollary} 

\begin{proof}
It follows easily from Theorem \ref{strengthNC}. In fact, the inequality in \eqref{strengthNCeq} implies that $\Lambda^{\#}\neq \emptyset$ and, therefore, $\Lambda^{\#} = \{\hat\lambda\}$ and \eqref{R0pos} necessarily holds.
\end{proof}


\section{Goh Transformation}\label{GohT}

In this section we introduce a linear transformation which is applied to the variable $(\xb,\ub,\vb),$ and that is motivated by the following. In the previous section we were able to provide a necessary condition involving the nonnegativity on $\C_2$ of the maximum of $\Omega\lam$ over the set $\Lambda^\#.$ The next step  is to find a sufficient condition and, in order to achieve this, one would usually strengthen the inequality \eqref{strengthNCeq} to convert it into a condition of strong positivity. But since no quadratic term on $\vb$ appears in $\Omega,$ the latter cannot be strongly positive. The technique we employ to find the desired sufficient condition is transforming $\Omega$ into a new quadratic mapping in some transformed variable, that may result strongly positive on an appropriate transformed critical cone. For historical interest, we recall that Goh introduced this transformation in \cite{Goh66a} and employed it to derive necessary conditions in \cite{Goh66a,Goh66}. Afterwards, Dmitruk in \cite{Dmi77} stated a second order sufficient condition for control-affine systems (case $l=0$) in terms of the uniform positivity of the maximum of $\Omega$ in the corresponding transformed space
of variables. 

Consider hence the linear system in \eqref{lineareq} and
the change of variables 
\be \label{Goht}
\left\{
\ba{l}
\yb_t:= \ds\int_0^t \vb_s {\rm d}s, \\
\xib_t := \xb_t-  F_{v,t}\,\yb_t, 
\ea
\right.
\quad {\rm for}\ t\in [0,T].
\ee
This change of variables can be done in any linear system of differential equations, and it is often called \textit{Goh's transformation}.
Observe that $\xib$ defined in that way satisfies the linear equation
\be \label{xieq}
\dot\xib_{t} = F_{x,t}\,\xib_{t} + F_{u,t}\,\ub_{t} +B_t\,\yb_{t},\quad
\xib_0=\xb_0,
\ee
where 
\be
\label{B1}
B_{t}:= F_{x,t} F_{v,t}-\ddt F_{v,t}.
\ee
The $i$th. column of $B$ is given by
\benl
-\sum_{j=0}^m \vh_j[f_i,f_j]^x + D_u f_i\, \dot{\uh},
\eenl
where $[f_i,f_j]^x:=({\rm D}_xf_i)f_j-(D_xf_j)f_i,$ and it is referred as the {\it Lie bracket with respect to $x$} of the vector fields $f_i$ and $f_j.$
We make the following hypothesis of regularity of the controls.
\begin{assumption}
\label{SmoothControls}
The controls $\uh$ and $\vh$ are smooth. 
\end{assumption}


\subsection{Tranformed critical cones}
In this paragraph we present the critical cones
obtained after Goh's transformation.
Recall the linearized endpoint constraints and linearized cost function in \eqref{linearconseq}-\eqref{linearconsineq}, and the critical cones given in \eqref{C2}-\eqref{C}.
Let $(\xb,\ub,\vb)\in \C$ be a critical
direction. Define $(\xib,\yb)$ by the transformation
\eqref{Goht} and set $\hb:=\yb_T.$
Note that \eqref{linearconseq}-\eqref{linearconsineq} yield
\begin{gather}
\label{tlinearconseq} D\eta_j (\xh_0,\xh_T)(\xib_0,\xib_T+F_{v,T}\hb)=0,\quad
\mr{for}\,\, j=1,\hdots,d_{\eta},
\\
\label{tlinearconsineq}
D\varphi_i(\xh_0,\xh_T)(\xib_0,\xib_T+F_{v,T}\hb)\leq
0,\quad \mr{for}\,\,
i=0,\hdots,d_{\varphi}.
\end{gather}
Recall the definition of the linear space $\W_2$ in paragraph \ref{ParCritical}.
Denote by $\Y$ the space $W^{1,\infty}(0,T;\cR^m),$
and consider the cones
\be 
\label{P}
\P:= \{(\xib, \ub,\yb,\hb)\in \W \times
\cR^m:\,\yb_0=0,\,\yb_T=\hb,\,\text{\eqref{xieq},
\eqref{tlinearconseq}-\eqref{tlinearconsineq} hold} \},
\ee
\be \label{P2}
\P_2:= \{(\xib, \ub,\yb,\hb)\in \W_2 \times
\cR^m:\,\text{\eqref{xieq}, \eqref{tlinearconseq}-\eqref{tlinearconsineq}
hold} \}.
\ee
\begin{remark}
 \label{PandC}
Note that $\P$ consists of the directions obtained by transformating the elements of $\C$ via \eqref{Goht}.
\end{remark}

The next result shows the density of $\P$ in $\P_2.$ This fact is useful afterwards to extend a necessary condition stated in $\P$ to the bigger cone $\P_2$ by continuity arguments, as it was done for $\C$ and $\C_2$ in Section \ref{SectionSOC}.
\begin{lemma} \label{PdenseP2}
$\P$ is a dense subset of $\P_2$ in the $ \W_2
\times \cR^m-$topology.
\end{lemma}

\begin{proof}
Note that the inclusion is immediate. In order to prove the density, consider the linear spaces
\begin{gather*}
 X:= \{ (\xib,\ub,\yb,\hb)\in \W_2\times \cR^m:\ \eqref{xieq}\ {\rm holds}\},\\
Y:=\{ (\xib,\ub,\yb,\hb)\in \W\times \cR^m:\ \yb(0)=0,\, \yb(T)=\hb\ {\rm and}\ \eqref{xieq}\ {\rm holds}\},
\end{gather*}
and the cone
\benl
Z:= \{ (\xib,\ub,\yb,\hb)\in X:\ \text{\eqref{tlinearconseq}-\eqref{tlinearconsineq}}\ {\rm holds}\}.
\eenl
Note that $Y$ is a dense linear subspace of $X$  (by  \cite[Lemma 6]{DmiShi10} or \cite[Lemma 8.1]{ABDL11}), and $Z$
is a finite-faced cone of $X. $ The desired density follows by Lemma \ref{lemmadense}. 
\end{proof}


\subsection{Transformed second variation}
Next we write the quadratic mapping $\Omega$ in the variables $(\xib,\ub,\yb,\vb,\hb).$
Set, for $\lambda\in\Lambda_L^{\#},$
\be 
\label{OmegaP}
\begin{split}
 \Omega_{\P}  &\lam (\xib,\ub,\yb,\vb,\hb) 
:= g\lam (\xib_0,\xib_T,\hb)
+ \ds \intT \left( \half\xib\,\tras H_{xx}\lam \xib + \ub\tras
H_{ux}\lam\xib \right.\\
 &\left. +\, \yb\tras M\lam \xib + \half\ub\tras H_{uu}\lam
\ub + \yb\tras E\lam \ub +
\half\yb\tras R\lam \yb + \vb\tras V\lam \yb
  \right) \dtt,
\end{split}
\ee 
where 
\begin{gather}
\label{M}
M:= F_v\tras H_{xx}-\dot H_{vx}-H_{vx}F_x,\quad E:= F_v\tras H_{ux}\tras - H_{vx}F_u,
\\
\label{SV}
S:=\half (H_{vx}F_v + (H_{vx}F_v)\tras),\quad
V:= \half (H_{vx}F_v - (H_{vx}F_v)\tras),
\\
\label{R1}
R := F_v\tras H_{xx}F_v - (H_{vx}B+(H_{vx}B)\tras) -
\dot S,
\\
\label{g}
g\lam (\zeta_0,\zeta_T,h):=
\half\ell''(\zeta_0,\zeta_T+F_{v,T}\,h)^2
+h\tras(H_{vx,T}\, \zeta_T+\half S_T h).
\end{gather}
Observe that, in view of Assumptions \ref{regular} and \ref{SmoothControls}, all the functions defined above are continuous in time. 
 
We can see that $M$ is a $m\times n-$matrix whose $i$th. row is given by 
\benl
M_i = p\sum_{j=0}^m \vh_j  \left(  \frac{\pr^2 f_j}{\pr x^2} f_i - 
 \frac{\pr^2 f_i}{\pr x^2} f_j +  \frac{\pr f_j}{\pr x} \frac{\pr f_i}{\pr x} - \frac{\pr f_i}{\pr x} \frac{\pr f_j}{\pr x} \right)- p \frac{\pr^2 f_i}{\pr x\pr u} \dot{\uh},
\eenl
$E$ is $m\times l$ with
$
E_{ij}=p\ds \frac{\pr^2 F}{\pr u_j \pr x} f_i - p  \frac{\pr f_i}{\pr x}\frac{\pr F}{\pr u_j},
$
the $m\times m-$matrices $S$ and $V$ have entries
$
S_{ij} = \ds\half p \left( \frac{\pr f_i}{\pr x}f_j + \frac{\pr f_j}{\pr x}f_i \right),
$
and
\be
\label{Vcrochet}
V_{ij}=p[f_i,f_j]^x.
\ee
The components of matrix $R$ have a quite long expression, that is simplified for some multipliers as it is detailed in the next section (see equation \eqref{Rij}).

\begin{theorem} 
\label{Omegat}
Let $(\xb,\ub,\vb) \in \H_2$ and $(\xib,\yb)$ defined by the transformation
\eqref{Goht}. Then
\benl
\Omega\lam (\xb,\ub,\vb) = \Omega_{\P} \lam
(\xib,\ub,\yb,\vb,\yb_T).
\eenl
\end{theorem}

\begin{proof}
First recall that the term $\vb\tras H_{vu}\lam \ub$ in $\Omega\lam$ vanishes since we are taking $\lambda\in \Lambda^{\#}$ and, in view of Lemma \ref{Lambdawlsc}, $H_{vu}\lam \equiv 0.$
In the remainder of the proof we omit the dependence on $\lambda$ for the sake of
simplicity.
Replacing $\xb$ in \eqref{Omega} by its expression in
\eqref{Goht} yields
\be
\label{J2}
\begin{split}
\Omega  (\xb,\ub,\vb) 
=&\, \half\ell''(\xh_0,\xh_T)(\xib_0,\xib_T+F_{v,T}\,\yb_T)^2
+ \ds \int_0^T \left[ \half(\xib+F_v\,\yb)\tras H_{xx}(\xib+F_v\,\yb) 
\right. \\
& \left. +\,\ub\tras H_{ux} (\xib+F_v\,\yb)  + \vb\tras H_{vx} (\xib+F_v\,\yb) + \half\ub\tras H_{uu}\,\ub \right] \dtt.
\end{split}
\ee
In view of \eqref{xieq} one gets
\be
\label{J2.2.1}
\int_0^T \vb\tras H_{vx}\, \xib \dtt
=[\yb\tras H_{vx}\, \xib]_0^T -\int_0^T \yb\tras
\{\dot{H}_{vx}\,\xib+H_{vx}(F_x\,\xib+F_u\,\ub+B\,\yb)\} \dtt.
\ee
The decomposition of $H_{vx}\,F_v$ introduced  in \eqref{SV} followed by an integration by parts leads to  
\be\label{J2.2.2}
\begin{split}
\int_0^T \vb\tras H_{vx}\,F_v \yb \dtt
&=\int_0^T \vb\tras (S+V) \yb \dtt\\
&=\half[\yb\tras S\yb]_0^T+\int_0^T(-\half \yb\tras
\dot S\yb + \vb\tras V\yb )\dtt.
\end{split}
\ee
The result follows by replacing using \eqref{J2.2.1} and \eqref{J2.2.2} in \eqref{J2}.
\end{proof}

Finally recall Theorem \ref{strengthNC}. Observe that by performing Goh's transformation in \eqref{strengthNCeq} and in view of Remark \ref{PandC}, we obtain the following form of the second order necessary condition.
\begin{corollary}
 \label{transNC}
 If $\wh$ is a weak minimum of problem (P), then
\be \label{maxOmegaP}
\max_{\lambda\in \Lambda^{\#}} \Omega_{\P} \lam
(\xib,\ub,\yb,\dot\yb,\yb_T) \geq 0,\quad \text{for all}\ (\xib,\ub,\yb,\yb_T)\in\P.
\ee
\end{corollary}


\section{New second order necessary condition}\label{SectionNC}

We aim to remove the dependence on $\vb$ in the formulation of the  second order necessary condition of Corollary \ref{transNC}.  
Note that in \eqref{maxOmegaP}, $\vb$ appears only in the term $\vb\tras V\lam \yb.$  Next we prove that we can
restrict the maximum in \eqref{maxOmegaP} to the
subset of $\Lambda^{\#}_L$ consisting of the
multipliers for which $V\lam$ vanishes.

Denote by ${\rm co}\, \Lambda^{\#}$ the convex hull of $\Lambda^{\#}$  and let $G({\rm co}\, \Lambda^{\#})$ be the subset of ${\rm co}\, \Lambda^{\#}$ for which $V\lam$ vanishes, i.e. 
\benl
G({\rm co}\, \Lambda^{\#}):=\{\lambda\in {\rm co}\, \Lambda^{\#}: V\lam \equiv 0\ \rm{on}\ [0,T]\}.
\eenl

The following optimality condition holds.
\begin{theorem}[New necessary condition]\label{newNC}
 If $\wh$ is a weak minimum of problem (P), then
\benl
\max_{\lambda\in G(\mr{co}\,\Lambda^{\#})} 
\Omega_{\P} \lam (\xib,\ub,\yb,\dot\yb,\yb_T) \geq
0,\quad \text{on}\ \P.
\eenl
\end{theorem}

\begin{remark}
It should be observe that one can have $0\in G({\rm co}\, \Lambda^{\#})$ and, in this case, the second order condition in Theorem \ref{newNC} does not provide any information. 
This situation may occur when the endpoint constraints are {\it not qualified,} in the sense of Assumption \ref{cq} introduced afterwards. This qualification is a natural generalization of the Mangasarian-Fromovitz \cite{ManFro67} condition to the infinite-dimensional framework.
\end{remark}

Theorem \ref{newNC} is an adaptation of similar results given in Dmitruk \cite{Dmi77} and Milyutin \cite{Mil81}, that were employed recently in Aronna et al. \cite{ABDL11}. The proof given in \cite[Theorem 4.6]{ABDL11} holds for Theorem \ref{newNC} with minor modifications and hence we do not include it in the present article.

Note that when $\wh$ has a unique associated multiplier, from Theorem \ref{newNC} we deduce that $G(\mr{co}\,\Lambda^{\#})$ is not empty and, since the latter is a singleton, we get the corollary below. This corollary is one of the necessary conditions stated by Goh in \cite{Goh66}.

\begin{corollary} 
\label{CoroCBsym}
 Assume that $\wh$ is a weak minimum having a unique associated multiplier. Then the following conditions holds. \begin{itemize} \item[(i)] $V\equiv 0 $ or, equivalently,  $H_{vx}F_v$ is symmetric or, in view of \eqref{Vcrochet},
\benl
p[f_i,f_j]^x=0,\quad \text{for}\ i,j=1,\dots,m,
\eenl
where $p$ is the unique associated adjoint state.
\item[(ii)] The matrix
\be
\label{R2}
\begin{pmatrix}
H_{uu} & E^\top \\
E & R
\end{pmatrix}
\ee
is positive semidefinite.
\end{itemize}

\end{corollary}

Observe that for $\lambda\in
G(\mr{co}\,\Lambda^{\#}),$ the quadratic form
$\Omega\lam$ does not depend on $\vb$ since
its coefficients vanish. We can then consider its
continuous extension to $\P_2,$ given by
\be 
\label{OmegaP2}
\begin{split}
\Omega_{\P_2}\lam&(\xib,\ub,\yb,\hb):= g\lam
(\xib_0,\xib_T,\hb)
+ \ds \intT \left( \half\xib\,\tras H_{xx}\lam \xib + \ub\tras
F_u\lam\xib \right.\\
 &\left. +\, \yb\tras M\lam \xib + \half\ub\tras H_{uu}\lam
\ub + \yb\tras J\lam \ub + \half\yb\tras R\lam \yb
\right) \dtt,
\end{split}
\ee
where the involved matrices and $g$ were defined in and \eqref{M}-\eqref{g}. Observe that, since $V\lam \equiv 0,$ one has that $H_{vx}\lam F_v$ is symmetric and, therefore, 
\be
\label{Rij}
R_{ij}\lam=-p\sum_{k=0}^m \vh_k [f_j,[f_k,f_i]^x]^x-p\left(  2\frac{\pr f_i}{\pr x} \frac{\pr f_j}{\pr u} +\frac{\pr f_j}{\pr x} \frac{\pr f_i}{\pr u}+\frac{\pr^2 f_i}{\pr u\pr x} f_j \right) \dot{\uh},
\ee
for each $i,j=1,\dots,m.$

From Theorem \ref{newNC} and Lemma \ref{PdenseP2}, it follows:
\begin{theorem} 
\label{NCP2}
 If $\wh$ is a weak minimum of problem (P), then
\be
\label{Omegapos}
\max_{\lambda\in G(\mr{co}\,\Lambda^{\#})} 
\Omega_{\P_2} \lam (\xib,\ub,\yb,\hb) \geq 0,\quad
\text{on}\ \P_2.
\ee
\end{theorem}

\begin{remark}
 The latter optimality condition does not involve $\vb.$ It is stated in the variable $(\xib,\ub,\yb,\hb).$
\end{remark}


\section[Second order sufficient condition]{Second order sufficient condition  for weak minimum}\label{SectionSC}

This section provides a second order sufficient
condition for strict weak optimality. 
Its proof is an adaptation of the
proof of \cite[Theorem 5.5]{ABDL11} with important simplifications due to the absence of control constraints, but with some new difficulties owed to the presence of the nonlinear control variable.

Define the  \textit{$\gamma-$order} by
\be 
\gamma_\P(\xb_0,\ub,\yb,\hb):=|\xb_0|^2+ |\hb|^2+
\intT(|\ub_t|^2+|\yb_t|^2)\dtt,
\ee
for $(\xb_0,\ub,\yb,\hb)\in \cR^n\times\U_2\times \V_2 \times
\cR^{m}.$
It can also be considered as a function of
$(\xb_0,\ub,\vb)\in \cR^n\times \U_2\times \V_2$ by setting
\be 
\gamma(\xb_0,\ub,\vb):=
\gamma_\P(\xb_0,\ub,\yb,\yb_T),
\ee
with $\yb$ being the primitive of $\vb$ defined in
\eqref{Goht}.

\begin{definition}\label{qgdef}[$\gamma-$growth]
We say that $\wh$ satisfies the
{\em $\gamma-$growth condition in the
weak sense} if there exist $\eps,\rho\gr 0$ such
that 
\be \label{qg}
J(w) \geq J(\wh) + \rho \gamma(x_0-\xh_0,u-\uh,v-\vh),
\ee 
for every feasible trajectory $w$ with $\|w-\wh\|_{\infty} \mi \eps.$
 
\end{definition}


\begin{theorem}[Sufficient condition for weak optimality]\label{SC}
\begin{itemize}
\item[(i)]
Assume that there exists $\rho \gr 0$ such that
\be
\label{unifpos}
\max_{\lambda\in G(\mr{co}\,\Lambda^{\#})} 
\Omega_{\P_2} \lam (\xib,\ub,\yb,\hb) \geq
\rho\gamma_\P(\xib_0,\ub,\yb,\hb),\quad \text{for all}\ (\xib,\ub,\yb,\hb)\in\P_2.
\ee
Then $\wh$ is a weak minimum satisfying 
$\gamma-$ growth in the weak sense.
\item[(ii)] Conversely, if $\wh$ is a weak solution satisfying $\gamma-$growth in the weak sense and such that $\alpha_0>0$ for every $\lambda\in G(\mr{co}\,\Lambda^{\#}),$ then \eqref{unifpos} holds for some $\rho>0.$

\end{itemize}
\end{theorem}

\begin{corollary}
\label{CoroSC}
If $\wh$ satisfies \eqref{unifpos} and it has a unique associated multiplier, then necessarily the matrix in \eqref{R2} is uniformly positive definite, i.e.
\be
\label{R22}
\begin{pmatrix}
H_{uu} & E^\top \\
E & R
\end{pmatrix}
\succeq \rho I,\quad \text{on} \ [0,T],
\ee
where $I$ refers to the identity matrix.
\end{corollary}

\begin{remark}
Another consequence of the condition \eqref{unifpos} is stated in Remark \ref{RemUP} afterwards, where we link it with the {\it strengthened generalized Legendre-Clebsch condition.}
\end{remark}

The remainder of this section is devoted to the proof of Theorem \ref{SC}. We shall start by establishing some technical results that will be needed for the main result. Recall first the following classical result for ordinary differential equations.

\begin{lemma}[Gronwall's Lemma]
\label{GronLem}
Let $a\in W^{1,1}(0,T;\cR^n),$ $b\in L^1(0,T)$ and $c\in L^1(0,T)$ be such that $|\dot{a}_t|\leq b_t + c_t |a_t|$ for a.a. $t\in (0,T).$ Then
\benl
\|a\|_\infty \leq e^{\|c\|_1}\big(|a_0|+\|b\|_1 \big).
\eenl
\end{lemma}

For the lemma below recall the definition of the space $\H_2$ in \eqref{H2}.

\begin{lemma}
\label{lemmaxbar}
There exists $\rho\gr 0$ such that 
\be
\label{xbargamma}
|\xb_0|^2+\|\xb\|_2^2+|\xb_T|^2\leq \rho
\gamma(\xb_0,\ub,\vb),
\ee
for every linearized trajectory $(\xb,\ub,\vb)\in\H_2.$ The constant $\rho$ depends on $\|A\|_{\infty},$
$\|F_v\|_{\infty},$ $\|E\|_{\infty}$ and
$\|B\|_{\infty}.$
\end{lemma}

\begin{proof}
Throughout this proof, whenever we put $\rho_i$ we refer to a positive constant depending on $\|A\|_{\infty},$
$\|F_v\|_{\infty},$ $\|E\|_{\infty},$ and/or
$\|B\|_{\infty}.$
Let $(\xb,\ub,\vb)\in\H_2$ and 
$(\xib,\yb)$ be defined by Goh's Transformation \eqref{Goht}.
 Thus $(\xib,\ub,\yb)$ is solution of \eqref{xieq}. Gronwall's Lemma and
Cauchy-Schwarz inequality yield
\be\label{lemmazxit}
\|\xib\|_{\infty}\leq \rho_1
(|\xib_0|^2+\|\ub\|_2^2+\|\yb\|_2^2)^{1/2}\leq
\rho_1 \gamma_\P(\xb_0,\ub,\yb,\yb_T)^{1/2},
\ee
with
$\rho_1=\rho_1(\|A\|_1,\|E\|_{\infty},\|B\|_{\infty}).$
This last inequality together with the relation between $\xib$ and $\xb$ provided by \eqref{Goht}
imply
\be\label{lemmazz}
\|\xb\|_2\leq \|\xib\|_2+\|F_v\|_{\infty}\|\yb\|_2\leq\rho_2
\gamma_\P(\xb_0,\ub,\yb,\yb_T)^{1/2},
\ee
for $\rho_2=\rho_2(\rho_1,\|F_v\|_{\infty}).$
On the other hand,  \eqref{Goht} and estimate
\eqref{lemmazxit} lead to
\benl
|\xb_T|\leq |\xib_T|+\|F_v\|_{\infty}|\yb_T|\leq
\rho_1 \gamma_\P(\xb_0,\ub,\yb,\yb_T)^{1/2}
+\|F_v\|_{\infty}|\yb_T|.
\eenl
Then, in view of Young's inequality `$2{ab}\leq {a^2+b^2}
$' for real numbers $a,b,$ one gets 
\be\label{lemmazzT}
|\xb_T|^2\leq
\rho_3 \gamma_\P(\xb_0,\ub,\yb,\yb_T),
\ee
for some
$\rho_3=\rho_3(\rho_1,\|F_v\|_{\infty}).$
The desired estimate follows from \eqref{lemmazz}
and \eqref{lemmazzT}.
\end{proof}

Note that Lemma \ref{lemmaxbar} above gives an estimate of the linearized state in the order $\gamma.$ The following result shows that the analogous property holds for the variation of the state variable as well. 

\begin{lemma} \label{lemmadeltax}
Given $C\gr 0,$ there exists $\rho\gr 0$ such that 
\benl
|\delta x_0|^2+\|\delta x\|^2_2+|\delta x_T|^2\leq \rho
\gamma(\delta x_0,\delta u,\delta v),
\eenl
for every $w=(x,u,v)$ solution of the state equation \eqref{stateeq} having $\|v\|_2\leq C,$ and where  $\delta w:=w-\wh.$ The constant $\rho$ depends on $C,$ $\|B\|_{\infty},$ $\|\dot B\|_{\infty}$ and the Lipschitz constants of $f_i.$
\end{lemma}

\begin{proof}
In order to simplify the notation we omit the dependence on $t.$
Consider $w=(x,u,v)$ solution of \eqref{stateeq} with
$\|v\|_2\leq C.$ Let $\delta w:=w-\wh,$ $\delta y_t:=\int_0^t \delta v_s {\rm d}s,$ and
$\xi:=\delta x-B\delta y,$ with $y_t:=\int_0^t v_s\mr{d}s.$ Note
that
\be \label{xidot}
\begin{split}
\dot\xi=
& \sum_{i=0}^m \left[v_if_i(x,u)-\vh_if_i(\xh,\uh)\right]-\dot B\delta y- \sum_{i=1}^m \delta
v_i\, f_i(\xh,\uh)\\
=&\sum_{i=0}^m v_i[f_i(x,u)-f_i(\xh,\uh)]-\dot{B}\delta y,\end{split}
\ee
where $v_0\equiv 1.$
In view of the Lipschitz-continuity of $f_i,$
\benl
|f_i(x,u)-f_i(\xh,\uh)|\leq L (|\delta x|+|\delta u|) \leq
L(|\xi|+\|B\|_{\infty}|\delta y|+|\delta u|),
\eenl
for some $L\gr 0.$
Thus, from \eqref{xidot} it follows
\benl
|\dot\xi|
\leq L(|\xi|+\|B\|_{\infty}|\delta
y|+|\delta u|)(1+|v|)+\|\dot B\|_{\infty}|\delta y|.
\eenl
Applying Gronwall's Lemma \ref{GronLem} one gets
\benl
\| \xi\|_\infty \leq e^{L\|1+|v|\,\|_1}\Big[ |\xi_0| + \left\|L(1+|v|)(\|F_v\|_\infty |\delta y| + |\delta u|)+\|\dot{F}_v\|_\infty |\delta y|\,\right\|_1  \Big].
\eenl
Note that the Cauchy-Schwarz inequality applied to previous estimate yields
\benl
\|\xi\|_{\infty} 
\leq \rho_1\big(|\xi_0|+\|\delta y\|_1+\|\delta u\|_1
+ \|\delta y\|_2\|v\|_2 +  \|\delta
u\|_2\|v\|_2\big),
\eenl
for $\rho_1=\rho_1(L,C,\|F_v\|_{\infty},\|\dot{F}_v\|_{\infty}).$
Since $\|\delta x\|_2\leq
\|\xi\|_2+\|F_v\|_{\infty}\|\delta y\|_2,$ by  previous estimate and Cauchy-Schwarz inequality,  the result follows.
\end{proof}

Finally, the following lemma gives an estimate for the difference between the variation of the state variable and the linearized state.

\begin{lemma}
\label{lemmaeta}
Consider $C\gr 0$ and $w=(x,u,v)\in\W$ a trajectory with $\|w-\wh \|_{\infty}\leq C.$ 
Set $(\delta x,\delta u,\delta v):= w-\wh$ and let $\xb$ be the linearization of $\xh$ associated with $(\delta x,\delta u,\delta v).$
Define
\benl
\vartheta:=\delta x-\xb.
\eenl
Then, $\vartheta$ is solution of the differential equation
\be
\label{doteta}
\begin{split}
\dot\vartheta
&= 
\sum_{i=0}^m \vh_i D_xf_{i} (\xh,\uh)\vartheta 
+ 
\sum_{i=1}^m \delta v_i Df_{i} (\xh,\uh)(\delta x,\ub) + \zeta,\\
\vartheta_0 &= 0,
\end{split}
\ee
where the remainder $\zeta$ is given by
\be 
\label{zeta}
\zeta:= \sum_{i=0}^m v_i \Big[\half D^2f_i(\xh,\uh)(\delta x,\ub)^2+ L|(\delta x,\ub)|^3\Big],
\ee
and $L$ is a Lipschitz constant for $D^2f_i,$ uniformly for $i=0,\dots,m.$ Furthermore, $\zeta$ satisfies the estimates
\be
\label{estzeta}
\|\zeta\|_{\infty} \mi \rho_1C,\quad \|\zeta\|_2 \mi \rho_1C\sqrt\gamma,
\ee
where $\rho_1= \rho_1(C,\|D^2 f\|_{\infty},L,\|v\|_\infty+1)$ and $\gamma:=\gamma(\delta x_0,\delta u,\delta v).$

If in addition, $C\rightarrow 0,$
the following estimates for $\vartheta$ hold
\be
\label{esteta}
\|\vartheta\|_{\infty} = o(\sqrt\gamma),
\quad 
\|\dot\vartheta\|_2 = o(\sqrt\gamma).
\ee
\end{lemma}

\begin{proof}
We shall note first that
\be
\label{dotdeltax1}
\dot{\delta x} =
\sum_{i=0}^m v_i \big[f_i(x,u)-f_i(\xh,\uh) \big] + \sum_{i=1}^m \delta v_{i}\,f_i(\xh,\uh).
\ee
Consider the following second order Taylor expansions for $f_i,$
\be
\label{Taylorfi}
f_i(x,u)= f_i(\xh,\uh) + D f_i(\xh,\uh)(\delta x,\delta u) + \half D^2f_i(\xh,\uh)(\delta x,\delta u)^2 + L|(\delta x,\delta u)|^3.
\ee
Combining \eqref{dotdeltax1} and \eqref{Taylorfi} yields
\be
\label{dotdeltax}
\dot{\delta x} =
\sum_{i=0}^m v_i D f_i(\xh,\uh)(\delta x,\delta u)+\sum_{i=1}^m \delta v_{i}f_i(\xh,\uh)+\zeta,
\ee
with the remainder being given by \eqref{zeta}.
The linearized equation \eqref{lineareq} together with \eqref{dotdeltax} lead to \eqref{doteta}.
In view of \eqref{zeta} and Lemma \ref{lemmadeltax}, it can be seen that the estimates in \eqref{estzeta} hold.

On the other hand, by applying Gronwall's Lemma \ref{GronLem} to \eqref{doteta}, and using Cauchy-Schwarz inequality afterwards lead to
\benl
\|\vartheta\|_{\infty} 
\leq 
\rho_3\left\|\sum_{i=1}^m \delta v_i D f_{i} (\xh,\uh) (\delta x,\delta u) + \zeta\right\|_1 
\leq 
\rho_4 \Big[ \|\delta v\|_2(\|\delta x\|_2 + \|\delta u\|_2) + \|\zeta\|_2 \Big],
\eenl
for some positive $\rho_3,\rho_4$ depending on $\|\vh\|_{\infty}$ and $\|Df\|_{\infty}.$ 
Finally, using the estimate in Lemma \ref{lemmadeltax} and \eqref{estzeta} just obtained, the inequalities in \eqref{esteta} follow.
\end{proof}

In view of Lemmas \ref{expansionlagrangian}, \ref{lemmaxbar}, \ref{lemmadeltax} and \ref{lemmaeta} we can justify 
the following technical result that is an essential point in the proof of the sufficient condition of Theorem \ref{SC}.

\begin{lemma}
\label{lemmasc2}
Let $w\in\W$ be a trajectory.
Set $(\delta x,\delta u,\delta v):= w-\wh,$ $\gamma:=\gamma(\delta x_0,\delta u,\delta v)$  and $\xb$ its corresponding linearized state, i.e. the solution of \eqref{lineareq}-\eqref{lineareq0} associated with $(\delta x_0,\delta u,\delta v).$ Assume that    $\|w-\wh\|_{\infty} \rightarrow 0.$
Then  
\be
\label{taylor0}
\L[\lambda](w) = \L[\lambda](\wh) 
+ \Omega[\lambda](\xb,\delta u,\delta v)+o(\gamma).
\ee
\end{lemma}

\begin{proof}
Omit the dependence on $\lambda$ for the sake of simplicity.
Recall the expansion of the Lagrangian function given in Lemma \ref{expansionlagrangian}.
Note that by Lemma \ref{lemmadeltax}, 
$
\L(w)
=
\L(\wh)+ \Omega(\delta x,\delta u,\delta v)+o(\gamma).
$
Hence, 
\be
\label{taylorlemma}
\L(w)
=
\L(\wh)+ \Omega(\xb,\delta u,\delta v)+\Delta\Omega + o(\gamma),
\ee
with 
$
\Delta\Omega:= \Omega(\delta x,\delta u,\delta v)-\Omega(\xb,\delta u,\delta v).
$
The next step is using Lemmas \ref{lemmaxbar}, \ref{lemmadeltax} and \ref{lemmaeta} to  prove that
\begin{equation}
\label{rest}
\Delta \Omega=o(\gamma).
\end{equation}
Note  that $\Q(a,a)-\Q(b,b)=\Q(a+b,a-b),$ for any bilinear mapping $\Q,$ and any pair $a,b$ of elements in its domain. Set $\vartheta:=\delta x-  \xb$ as it is done in Lemma \ref{lemmaeta}. 
Hence, 
\benl
\begin{split}
\Delta \Omega
= &\,
\half\ell'' ( (\delta x_0+\xb_0,\delta x_T+\xb_T) , (0,\vartheta_T) ) \\
& +\ds\int_0^T [\half(\delta x+\xb)\tras Q\vartheta + \delta u\tras E\vartheta + \delta v\tras C\vartheta   ] \dtt.
\end{split}
\eenl
The estimates in Lemmas \ref{lemmaxbar}, \ref{lemmadeltax} and \ref{lemmaeta} yield
$\Delta \Omega = \intT \delta v\tras C\vartheta  \dtt + o(\gamma).$
Integrating by parts in the latter expression and using \eqref{esteta} leads to
\benl
\intT \delta v\tras C\vartheta  \dtt 
=
[\yb\tras C\vartheta]_0^T - \intT \yb\tras(\dot {C}\vartheta + C \dot\vartheta  )\dtt
= o(\gamma), 
\eenl
and hence the desired result follows.
\end{proof}


\begin{proof}
[of Theorem \ref{SC}]
We shall prove that if \eqref{unifpos} holds for
some $\rho>0,$ then $\wh$ satisfies
$\gamma-$growth in the weak sense.
By the contrary assume that the
$\gamma-$growth condition \eqref{qg} is not satisfied.
Consequently, there exists
a sequence of feasible trajectories $\{w_k\}$ converging to $\wh$ in the weak sense, such that
\begin{equation}
\label{qgrowth}
J(w_k)\leq J(\wh)+o(\gamma_k),
\end{equation}
with 
\benl
(\delta x_k,\ub_k,\vb_k):= w_k-\wh \,\,\, \text{and} \,\,\,
\gamma_k:= \gamma (\delta x_{k,0},\ub_k,\vb_k).
\eenl
Let $(\xib_k,\ub_k,\yb_k)$ be the transformed directions defined by \eqref{Goht}. We divide the remainder of the proof in two steps.
 \begin{itemize}
\item[(I)] First  we prove that the sequence given by 
\be
(\tilde{\xi}_k,\tilde{u}_k,\tilde{y}_k,\tilde{h}_k):= (\xib_k,\ub_k,\yb_k,\hb_k)/{\sqrt{\gamma_k}}
\ee
 contains a subsequence converging to an element $(\tilde\xi,\tilde u,\tilde y,\tilde h)$ of $\P_2$ in the weak topology, i.e. $(\tilde u_k,\tilde y_k)\rightharpoonup (\tilde u,\tilde y)$ in the weak topology of $\U_2\times \V_2$ and $(\tilde{\xi}_k,\tilde{h}_k)\rightarrow (\tilde\xi,\tilde h)$ in the strong sense of $\X_2\times \cR^m.$
\item[(II)] 
Afterwards, employing the latter sequence and its weak limit, we show that \eqref{unifpos} together with \eqref{qgrowth} lead to a contradiction.
\end{itemize}


We shall begin by {Part (I).} For this we take an arbitrary Lagrange multiplier $\lambda$ in $\Lambda_L^{\#}.$
By multiplying the inequality \eqref{qgrowth} by $\alpha_0,$ and adding the nonpositive term
\be
\sum_{i=1}^{d_{\varphi}}\alpha_i\varphi_i(x_{k,0},x_{k,T})+\sum_{j=1}^{d_{\eta}}\beta_j\eta_j(x_{k,0},x_{k,T}),
\ee
to its left-hand side, the inequality follows
\begin{equation}
\label{quadlag}
\L[\lambda](w_k)\leq\L[\lambda](\wh)+o(\gamma_k).
\end{equation}
Let us now recall the expansion \eqref{taylor0} given in Lemma \ref{lemmasc2}. Note that the elements of the sequence $(\tilde\xi_{k,0},\tilde u_k,\tilde y_k,\tilde h_k)$ have unit $\cR^n\times \U_2\times \V_2\times \cR^m-$norm. The Banach-Alaoglu Theorem (see e.g. \cite[Theorem III.15]{Bre83})  implies that, extracting if
necessary a subsequence, there exists
$(\tilde\xi_0,\tilde u,\tilde y,\tilde h)\in \cR^n\times \U_2\times \V_2\times \cR^m$ 
such that
\begin{equation}
\label{limityk}
\tilde\xi_{k,0}\rightarrow \tilde\xi_0,\quad
\tilde u_k\rightharpoonup \tilde u,\quad
\tilde y_k\rightharpoonup \tilde y,\quad
\tilde h_k\rightarrow \tilde h,
\end{equation}
where the two limits indicated with
$\rightharpoonup$ are taken in the weak topology of
$\U_2$ and $\V_2,$ respectively.
The solution of equation \eqref{xieq} associated with
$(\tilde \xi_0,\tilde u,\tilde y)$ is denoted by $\tilde\xi,$ which is the limit of $\tilde\xi_k$ in $\X_2.$ 
For the aim of proving that $(\tilde\xi,\tilde u,\tilde v,\tilde h)$ belongs to $\P_2,$ we shall check that the initial-final conditions \eqref{tlinearconseq}-\eqref{tlinearconsineq} are verified.
For each index $0\leq i\leq d_{\varphi},$ one has
\be 
\label{phineg'}
D\varphi_i(\xh_0,\xh_T) (\tilde\xi_0,\tilde\xi_T+B_T\tilde h)
=  
\lim_{k\rightarrow \infty} D\varphi_i(\xh_0,\xh_T)\left(\frac{\xb_{k,0},\xb_{k,T}}{\sqrt{\gamma_k}}\right).
\ee
In order to prove that the right hand-side of \eqref{phineg'} is nonpositive, we consider the following first order Taylor expansion of
function $\varphi_i$ around $(\xh_0,\xh_T):$ 
\benl
\varphi_i(x_{k,0},x_{k,T})
= 
\varphi_i(\xh_0,\xh_T) + D \varphi_i(\xh_0,\xh_T) (\delta x_{k,0}, \delta x_{k,T}) 
+o(|(\delta x_{k,0}, \delta x_{k,T})|).
\eenl
Previous equation and Lemmas \ref{lemmaxbar} and \ref{lemmaeta} imply
\benl
\label{taylor_phi}
\varphi_i(x_{k,0},x_{k,T})=\varphi_i(\xh_0,\xh_T)+D\varphi_i(\xh_0,\xh_T)(\xb_{k,0},\xb_{k,T})+o(\sqrt{\gamma_k}).
\eenl
Thus, the following approximation for the right hand-side in \eqref{phineg'} holds,
\be
\label{difphi}
D\varphi_i(\xh_0,\xh_T) \left( \frac{\xb_{k,0},\xb_{k,T}}{\sqrt{\gamma_k}} \right)
=\frac{\varphi_i(x_{k,0},x_{k,T})-\varphi_i(\xh_0,\xh_T)}{\sqrt{\gamma_k}}+o(1).
\ee
Since $w_k$ is a feasible trajectory, it satisfies \eqref{finalineq} and, therefore, equations
\eqref{phineg'} and \eqref{difphi} yield, for $1\leq i\leq d_{\varphi},$
$
D\varphi_i(\xh_0,\xh_T)(\tilde\xi_0,\tilde\xi_T+B_T\tilde h)\leq 0.
$ 
For $i=0$ use inequality \eqref{qgrowth} to get the corresponding inequality.
Analogously, 
\be
\label{eta0}
D\eta_j(\xh_0,\xh_T)(\tilde\xi_0,\tilde\xi_T+B_T\tilde h) = 0,\quad \mr{for}\ j=1,\hdots,d_{\eta}.
\ee
Thus $(\tilde\xi,\tilde u,\tilde y,\tilde h)$ satisfies \eqref{tlinearconseq}-\eqref{tlinearconsineq}, and hence it belongs to $\P_2.$


Let us deal with {Part (II).} 
Note that from \eqref{taylor0} and \eqref{quadlag} we get
 \be
\Omega_{\P_2} \lam (\tilde\xi_k,\tilde u_k,\tilde y_k,\tilde h_k) \leq o(1),
\ee
and thus
\be 
\label{lims}
\liminf_{k\rightarrow \infty}\, \Omega_{\P_2} \lam (\tilde\xi_k,\tilde u_k,\tilde y_k,\tilde h_k) \leq 0.
\ee
Consider the subset of $G({\rm co}\,\Lambda_L^{\#})$ defined by
\be
\Lambda_L^{\#,\rho}:= \{\lambda\in G({\rm co}\,\Lambda_L^{\#}): \Omega_{\P_2}\lam - \rho\gamma_\P\ {\rm is}\ {\rm weakly}\,{\rm l.s.c.}\,{\rm on}\ \H_2\times\cR^m \}.
\ee
By applying Lemma \ref{quadform} to the inequality \eqref{unifpos} one has
\be
\label{maxLamrho}
\max_{\lambda\in \Lambda^{\#,\rho}_L} 
\Omega_{\P_2} \lam (\xib,\ub,\yb,\hb) \geq
\rho\gamma_\P(\xib_0,\ub,\yb,\hb),\quad \text{on}\
\P_2.
\ee
We shall take $\tilde\lambda\in \Lambda^{\#,\rho}_L$ that attains the maximum in \eqref{maxLamrho} for the direction $(\tilde \xi,\tilde u,\tilde y,\tilde h).$ Hence we get
\be
\ba{rl}
\vspace{3pt}0
&\leq 
\Omega_{\P_2}[\tilde\lambda](\tilde \xi,\tilde u,\tilde y,\tilde h) - \rho\gamma_\P(\tilde \xi_0,\tilde u,\tilde y,\tilde h)\\
&\leq
\liminf_{k\rightarrow \infty} \Omega_{\P_2}[\tilde\lambda](\tilde\xi_k,\tilde u_k,\tilde y_k,\tilde h_k) - \rho\gamma_\P(\tilde\xi_{k,0},\tilde u_k,\tilde y_k,\tilde h_k)
\leq
-\rho,
\ea
\ee
since  $\Omega_{\P_2}[\tilde\lambda] - \rho\gamma_\P$ is weakly-l.s.c., $\gamma_\P(\tilde\xi_{k,0},\tilde u_k,\tilde y_k,\tilde h_k)=1$ for every $k$
and inequality \eqref{lims} holds. 
This leads us to a contradiction since $\rho\gr0.$ Therefore, the desired result follows.


(ii) Let us now prove the second statement. Assume that $\wh$ is a weak solution satisfying $\gamma-$growth in the weak sense for some constant $\rho'>0,$ and such that $\alpha_0>0$ for every multiplier $\lambda \in G({\rm co}\,\Lambda^{\#}).$ We consider the modified problem
\be\label{tildeP}\tag{$\tilde P$}
\min \{ J(w)-\rho' \gamma(x_0-\xh_0,w-\wh) :  \text{\eqref{stateeq}-\eqref{finalineq}} \},
\ee
and rewrite it in the Mayer form
\be\label{2tildeP}\tag{$\tilde{\tilde P}$}
\begin{split}
 &J(w)-\rho' (|x_0-\xh_0|^2 + |y_T-\yh_T|^2 +\pi_{1,T}+\pi_{2,T}) \rightarrow \min,\\
 &\text{\eqref{stateeq}-\eqref{finalineq}},\\
& \dot y=v,\\
& \dot\pi_1 = (u-\uh)^2,\\
& \dot \pi_2 = (y-\yh)^2,\\
& y_0=0,\ \pi_{1,0}=0,\ \pi_{2,0}=0.
\end{split}
\ee
We aim to apply the second order necessary condition of Theorem \ref{NCP2} to \eqref{2tildeP} at the point $(w=\wh,y=\yh,\pi_1=0,\pi_2=0).$ Simple computations show that at this solution each critical cone of \eqref{P} is the projection of the corresponding critical cone of \eqref{2tildeP}, and that the same holds for the set of multipliers. Furthermore, the second variation of \eqref{2tildeP} evaluated at a multiplier $\tilde{\tilde \lambda} \in G ({\rm co}\, \tilde{\tilde \Lambda}^{\#})$ is given by
\be
\Omega_{\P_2}\lam (\xib,\ub,\yb,\yb_T) - \alpha_0 \rho'\gamma_\P(\xb_0,\ub,\yb,\yb_T),
\ee
where $\lambda \in G ({\rm co}\, {\Lambda}^{\#})$ is the corresponding multiplier for problem \eqref{P}. 
Hence, the necessary condition in Theorem \ref{NCP2} (see Remark \ref{NCt} below) implies that for every $(\xib,\ub,\vb,\hb)\in \P_2$ there exists $\lambda \in G ({\rm co}\, {\Lambda}^{\#})$ such that
\benl
\Omega_{\P_2}\lam (\xib,\ub,\yb,\yb_T) - \alpha_0 \rho'\gamma_\P(\xb_0,\ub,\yb,\yb_T) \geq 0.
\eenl
Setting $\rho:= \min_{G ({\rm co}\, {\Lambda}^{\#})} \alpha_0 \rho' >0$ yields the desired result.
This completes the proof of the theorem.
\end{proof}

\begin{remark}\label{NCt}
Actually, since the dynamics of \eqref{2tildeP} are not autonomous, we apply an extension of Theorem \ref{NCP2} to time-dependent dynamics. The latter follows  by adding a state variable $\kappa$ with dynamics $\dot\kappa = 1$ and $\kappa_0=0,$ and obtaining an autonomous formulation. 
\end{remark}


\section{Shooting algorithm}\label{SectionShoot}

The purpose of this section is to present an appropriate numerical scheme to solve the problem given by \eqref{cost}-\eqref{finaleq}. 
Note that {no inequality endpoint constraints are considered.} 
More precisely, we investigate the formulation and the convergence of an algorithm that approximates an optimal solution provided an initial estimate.

We shall then assume that $\wh$ is a weak solution for \eqref{cost}-\eqref{finaleq}, and let us consider an hypothesis concerning the endpoint conditions. With this end recall Definition \ref{DefCQ}.
The following holds throughout the rest of the article.

\begin{assumption}
\label{cq}
The endpoint equality constraints are qualified at $\wh$ or, equivalently, the derivative of $\bar\eta$ at $(\xh_0,\uh,\vh)$ is onto.
\end{assumption}

It is a well-known result that in this case $\wh$ is normal and has a unique associated multiplier  (see e.g. Pontryagin et al. \cite{PBGM}).
Therefore, without loss of generality, we can consider $\alpha_0=1.$ The unique multiplier associated with $\wh$ is denoted by $\hat\lambda=(\hat\beta,\ph).$

\subsection{Optimality System} In what follows we use the first order optimality conditions \eqref{stationarity} to provide a set of equations from which we can determine $\wh.$ We obtain an optimality system in the form of a {\it two-point boundary value problem} (TPBVP).

For the sake of simplicity of the presentation, we omit the dependence on $t.$ We assume as well that the expressions hold everywhere on $[0,T],$ at least that it is specified otherwise, and that whenever some argument of $f_i,$ $H,$ $\ell,$ $\L$ or their derivatives is omitted,  they are evaluated at $(\wh,\hat\lambda).$  Recall the Assumptions \ref{regular} and \ref{SmoothControls}, that are supposed to hold true here.

\subsubsection{Totally nonlinear case} We shall recall that for the case where all the control variables appear nonlinearly ($m=0$), the classical technique is using the stationarity equation 
\be 
\label{Hu0}
H_u\lamh (\wh)=0,
\ee
to write $\uh$ as a function of $(\xh,\hat\lambda)$ (this is done in e.g. \cite{Bul71,MauGil75,BonCai06,Tre12}). One is able to do this by assuming, for instance, the \textit{strengthened Legendre-Clebsch condition} 
\be
\label{strongLC}
H_{uu}\lamh (\wh) \succ 0.
\ee
The latter condition comes from strengthening the inequality in the necessary optimality condition mentioned in Remark \ref{remarkLC}, which is verified by $\wh$ in view of Corollary \ref{NCunique}.
In this case, due to the Implicit Function Theorem, we can write $\uh=U\lamh(\xh)$ with $U$ being a smooth function. Hence, replacing the occurrences of $\uh$  by $U\lamh(\xh)$ in the  state and costate equations yields a two-point boundary value problem.

\subsubsection{Totally affine case} On the other hand, when the system is affine in all the control variables ($l=0$), we cannot eliminate the control from the equation $H_v=0$ and, therefore, a different technique is employed (see e.g. \cite{Mau76,Obe79,Pes94,BonCai06,ABM11,Tre12}).  
The idea is to consider an index $1\leq i \leq m,$ and to take ${d^{M_i}H_v}/{dt^{M_i}}$ to be the lowest order derivative of $H_v$ in which $\vh_i$ appears with a coefficient that is not identically zero.
Kelley \cite{Kel64}, Goh \cite{Goh66a,GohThesis}, Kelley et al. \cite{KelKopMoy67} and Robbins \cite{Rob67} proved that $M_i$ is even when the investigated extremal is normal.
This implies that $\dot H_v$ depends only on $\xh$ and $\hat\lambda$ and, consequently, it is differentiable in time. Thus the expression
\be
\label{ddotHv0}
\ddot{H}_v \lamh (\wh)=0
\ee
is well-defined. The control $\vh$ can be retrieved from \eqref{ddotHv0}  provided that, for instance, the \textit{strengthened generalized Legendre-Clebsch condition} 
\be
\label{strongenLC}
-\frac{\partial \ddot H_v}{\partial v}\lamh (\wh) \succ 0
\ee
holds (see Goh \cite{GohThesis,Goh95,Goh08}). 
In this case, we can write $\vh =  V\lamh(\xh)$ with $V$ being differentiable.
By replacing $\vh$ by $V\lamh(\xh)$ in the state-costate equations, we get an optimality system in the form of a boundary value problem.

\subsubsection{Partially affine case} In the problem studied here, where $l\gr 0$ and $m\gr 0,$ we aim to use both equations \eqref{Hu0} and \eqref{ddotHv0} to retrieve the control $(\uh,\vh)$ as a function of  the state $\xh$ and the multiplier $\hat\lambda.$
We describe next a procedure to achieve this elimination
that was proposed in Goh \cite{Goh95,Goh08}. 

We shall start by proving that $H_v$ can be differentiated twice in the time variable, as it was done for the totally affine case.
Observe that \eqref{Hu0} may be used to write $\dot\uh$ as a function of $(\hat\lambda,\wh).$ 
In fact, in view of Corollary \ref{NCunique},
\be
\label{Huv}
H_{uv}\lamh(\wh)=0,
\ee
and hence the coefficient of $\dot\vh$ in $\dot{H}_u$ is zero. Consequently,
\be 
\label{dotHu0}
\dot H_u=\dot H_u\lamh (\xh,\uh,\vh,\dot\uh)=0,
\ee
and, if the strengthened Legendre-Clebsch condition \eqref{strongLC} holds, $\dot{\uh}$ can be eliminated from \eqref{dotHu0} yielding
\be 
\label{dotu}
\dot\uh = \Gamma\lamh(\xh,\uh,\vh).
\ee
Take now an index $i=1,\dots,m$ and observe that
\be\label{dotHv}
0=\dot{H}_{v_i} = \ddt \,\ph \fh_i = \ph\sum_{j=0}^m \vh_j [f_j,f_i]^x (\xh,\uh)+{H}_{v_iu}\dot{\uh}=\hat{p}\, [f_0,f_i]^x(\xh,\uh),
\ee
where Corollary \ref{CoroCBsym} and \eqref{Huv} are used in the last equality. Therefore, 
$
\dot H_v = \dot H_v \lamh (\xh,\uh).
$
We can then differentiate one more time $\dot{H}_v,$ replace the occurrence of $\dot\uh$ by $\Gamma$ in \eqref{dotu} and obtain \eqref{ddotHv0} as it was desired.
See that \eqref{ddotHv0} together with the boundary conditions
\begin{gather}
 \label{HvT}  H_v\lamh (\wh_T)=0,\\
 \label{dotHv0} \dot H_v \lamh (\wh_0) = 0,
\end{gather}
guarantee the second identity in the stationarity condition \eqref{stationarity}.

\noindent \textbf{Notation:} Denote by (OS) the set of equations consisting of \eqref{stateeq}-\eqref{finaleq}, \eqref{nontriv}, \eqref{costateeq}-\eqref{transvcond}, \eqref{Hu0},  \eqref{ddotHv0} and the boundary conditions \eqref{HvT}-\eqref{dotHv0}.

\begin{remark}
 Instead of \eqref{HvT}-\eqref{dotHv0}, we could choose another pair of endpoint conditions among the four possible ones: $H_{v,0}=0,$ $H_{v,T}=0,$ $\dot H_{v,0}=0$ and $\dot H_{v,T}=0,$ always including at least one of order zero. The choice we made will simplify the presentation of the result afterwards.
\end{remark}

Observe now that the derivative with respect to $(u,v)$ of the mapping
$
(w,\lambda) \mapsto
\begin{pmatrix}
 H_u\lam(w)
\\
-\ddot H_v\lam(w)
\end{pmatrix}
$
at $(\wh,\hat\lambda)$ is  given by
\be 
\label{Jac}
\J:=
\begin{pmatrix}
 H_{uu}\lamh(\wh) & H_{uv}\lamh(\wh)
\\
-\ds\frac{\partial \ddot{H}_v}{\partial u} \lamh(\wh) & 
-\ds\frac{\partial \ddot{H}_v}{\partial v}\lamh(\wh)
\end{pmatrix}.
\ee
Since \eqref{Huv} holds along $(\wh,\hat\lambda),$ whenever \eqref{strongLC} and \eqref{strongenLC} are verified, the matrix $\J$ is definite positive and consequently, nonsingular. In this case we may write $\uh=U\lamh (\xh)$ and $\vh=V\lamh (\xh)$ from  \eqref{Hu0} and \eqref{ddotHv0}. Thus (OS) can be regarded as a TPBVP whenever the following hypothesis is verified.

\begin{assumption}
\label{assumptionLC}
The conditions \eqref{strongLC} and \eqref{strongenLC} hold along $\wh.$
\end{assumption}

Summing up we get the following result.
\begin{proposition}[Elimination of the control]
\label{ElimCont}
If $\wh$ is a weak solution for which Assumption \ref{assumptionLC} holds, then one has
\benl
\uh=U\lamh (\xh), \quad \vh=V\lamh (\xh),
\eenl
for smooth functions $U$ and $V.$
\end{proposition}
\begin{remark}
When the linear and nonlinear controls are uncoupled, this elimination of the controls is much simpler. An example is shown in Oberle \cite{Obe90} where a nonlinear control variable can be eliminated by the stationarity of the pre-Hamiltonian, and the remaining problem has two uncoupled controls, one linear and one nonlinear. 
\end{remark}

The rest of this article is very close to what was done in Aronna et al. \cite{ABM11} for the totally affine case. The main difference between the study in \cite{ABM11} and the mixed case treated here lies on the derivation of the system (OS). The proof of the convergence in Section \ref{SectionWP} is an extension of the proof of Theorem 5 in \cite{ABM11}.  The presentation here is then more concise, and the reader is referred to the mentioned article for further details.

\subsection{The algorithm}

The aim of this section is to present a numerical scheme to solve system (OS). In view of Proposition \ref{ElimCont} we can define the following mapping.
\begin{definition}
Let $\shoot:\cR^n\times\cR^{n+d_{\eta},*} =: {\rm D}(\mathcal{S})  \rightarrow \ \cR^{d_{\eta}}\times \cR^{2n+2m,*}$ be the {\em shooting function} given by
\be
\label{shoot1}
\begin{array}{rl}
\begin{pmatrix} 
x_0,p_0,\beta 
\end{pmatrix}
 =:\nu
&\mapsto
\ \shoot (\nu):=
\begin{pmatrix}
\eta(x_0,x_T)\\
p_0 +D_{x_0}\ell\lam(x_0,x_T)\\
p_T -D_{x_T}\ell\lam(x_0,x_T)\\
H_v\lam(w_T)\\
\dot H_v (w_0)
\end{pmatrix},
\end{array}
\ee
where $(x,p)$ is a solution of  \eqref{stateeq},\eqref{costateeq},\eqref{Hu0},\eqref{ddotHv0} 
with initial conditions $x_0$ and $p_0,$ and $\lambda:=(p,\beta),$ and where the occurrences of $u$ and $v$ were replaced by $u=U[\lambda](x)$ and $v=V[\lambda](x).$
\end{definition}

Note that solving (OS) consists of finding $\hat\nu \in {\rm  D}(\mathcal{S})$ such that
\be\label{S=0}
\mathcal{S}(\hat\nu)=0.
\ee
Since the number of equations in \eqref{S=0} is greater than the number of unknowns, the Gauss-Newton method is a suitable approach to solve it. The \textit{shooting algorithm} we propose here consists of solving the equation \eqref{S=0} by the Gauss-Newton method.
This algorithm solves, at each iteration $k$ with corresponding value $\nu^k,$ the linear approximation of
\be
\label{chap2minnorm}
\min_{\Delta\in {\rm  D}(\mathcal{S})} \left|\mathcal{S}(\nu^k)+
\mathcal{S}'(\nu^k)\Delta\right|^2,
\ee
obtaining a solution $\Delta^k.$ Afterwards, it updates
$
\nu^{k+1}\leftarrow \nu^k+\Delta^k.
$
Note that in order to solve the linear approximation of problem \eqref{chap2minnorm} we look for $\Delta^k$ in the kernel of the derivative of the objective function, i.e. $\Delta^k$ satisfying
\benl
\label{chap2eqS}
\mathcal{S}'(\nu^k)\tras \mathcal{S}'(\nu^k)\Delta^k + \mathcal{S}'(\nu^k)\tras \mathcal{S}(\nu^k)=0.
\eenl
Hence, to calculate $\Delta^k$ from previous equation, the matrix $\mathcal{S}'(\nu^k)\tras \mathcal{S}'(\nu^k)$ must be nonsingular.
Thus, the Gauss-Newton method is applicable provided that
$\mathcal{S}'(\hat\nu)$ is one-to-one, with $\hat\nu:=(\xh_0,\ph_0,\hat\beta).$ Furthermore, since the right hand-side of system \eqref{S=0} is zero, it converges locally quadratically if the function $\shoot$ has Lipschitz continuous derivative.
The latter holds true here given the regularity hypotheses on the data functions (in Assumption \ref{regular}).
This convergence result is stated in the proposition below. See e.g. Fletcher \cite{Fle80} for a proof.
\begin{proposition}
\label{Conv}
  If $\mathcal{S}'(\hat\nu)$ is one-to-one then the shooting algorithm is locally quadratically convergent.
\end{proposition}


\section{Convergence of the shooting  algorithm: application of the second order sufficient condition}\label{SectionWP}{}
The main result of this last part of the article is the theorem below that gives a condition guaranteeing the quadratic convergence of the shooting method near an optimal local solution.

\begin{theorem}
\label{wp} 
Suppose that $\wh$ verifies Assumption \ref{assumptionLC} and the condition \eqref{unifpos}. Then the shooting algorithm is locally quadratically convergent.
\end{theorem}

The idea is to link the sufficient condition \eqref{unifpos} to the derivative $\mathcal{S}'(\hat\nu).$ Note that \eqref{unifpos} is expressed in the variables after Goh's Transformation, while $\shoot$ is in the original variables. The procedure to achieve Theorem \ref{wp} has three stages that are described in the paragraphs \ref{LinOS}, \ref{SecLQ} and \ref{SecTransf} below. The proof of Theorem \ref{wp} is at the end of the subsection \ref{SecTransf}. 

\begin{remark}
\label{RemUP}
In view of a result in Goh \cite[Section 4.8]{GohThesis} the positive definiteness in \eqref{R22} implies both \eqref{strongLC} and \eqref{strongenLC}. Therefore, in Theorem \ref{wp}, the Assumption \ref{assumptionLC} is guaranteed by the condition \eqref{unifpos}.
Nevertheless, the computations linking \eqref{R22} and Assumption \ref{assumptionLC} are long and not trivial and, therefore, we do not include them in this article.
\end{remark}


\if{
Step 1. Use the linearization of (OS) to obtain a system that models the derivative of $\shoot.$ Call $(LS)$ this linearized system. Step 2. Define an auxiliary linear-quadratic problem (LQ) in the transformed variables, and write (LQS) the first order optimality system of (LQ). Observe that \eqref{unifpos} is a sufficient condition for optimality of (LQ) and thus (LQS) has a unique solution equal to 0. Step 3. Show that each solution of (LS) can be transformed into a solution of (LQS). Conclude that \eqref{unifpos} implies that (LS) has a unique solution equal to 0, which yields the Theorem \ref{wp}.
}\fi


\subsection{Linearization of (OS)}
\label{LinOS}
We write the linearized system associated with (OS), which gives the derivative of $\shoot.$ 
A definition of linearized differential algebraic system can be found in e.g. Kunkel-Mehrmann  \cite{KunMeh06} or Aronna et al. \cite{ABM11}. 
We denote by ${\rm  Lin}\,\F$  the {\it linearization} of function $\F,$ i.e.
\benl
{\rm  Lin}\,\F\mid_{(\zeta^0_t,\alpha^0_t)}(\bar{\zeta}_t,\bar{\alpha}_t):=\F'(\zeta^0_t,\alpha^0_t)(\bar{\zeta}_t,\bar{\alpha}_t),
\eenl
The technical result below will simplify the computation afterwards. Its proof is immediate (or see \cite{KunMeh06}).

\begin{lemma}[Commutation of linearization and differentiation]
\label{lindiff}
Given $\G$ and $\F$ as in the previous definition, it holds:
\be \label{lemmaLineq}
\ddt\,{\rm  Lin}\, \G={\rm  Lin}\,\ddt \G,\quad 
\ddt\,{\rm  Lin}\, \F={\rm  Lin}\,\ddt \F.
\ee 
\end{lemma}
Note that, since $H_v=pB,$ 
\be
\label{LinHv} 
{\rm Lin}\,H_v = \pb F_v + \xb\tras H_{vx}\tras.
\ee
Here, whenever the argument of a function is missing, assume that it is evaluated on $(\wh,\hat\lambda).$ 
The linearization of system (OS) at point $(\xh,\uh,\vh,\hat\lambda)$ consists of the linearized state equation \eqref{lineareq}
with endpoint condition \eqref{linearconseq},
the linearized costate equation
\be
\label{linearcostate}
 -\dot \pb_t 
= 
\pb_tF_{x,t} + \xb_t\tras H_{xx,t} + \ub_t\tras H_{ux,t} + \vb_t\tras H_{vx,t},\quad {\rm  a.e.}\ {\rm  on}\ [0,T],
\ee
with boundary conditions
\begin{align}
\label{condq0} 
  \pb_0 
&= 
-\left[\xb_0\tras D^2_{x_0^2}\ell
+\xb_T\tras D^2_{x_0x_T}\ell 
+\sum_{j=1}^{d_{\eta}}{\bar\beta}_jD_{x_0}\eta_j\right]_{(\xh_0,\xh_T)},\\
\label{condqT}
\pb_T&=\left[\xb_T\tras D^2_{x_T^2}\ell+\xb_0\tras D^2_{x_0x_T}\ell 
+\sum_{j=1}^{d_{\eta}}{\bar\beta}_jD_{x_T}\eta_j\right]_{(\xh_0,\xh_T)},
\end{align}
and the algebraic equations
\begin{eqnarray}
 \label{LinHu0} 
0 &=& {\rm  Lin}\ H_u = \pb F_u + \xb\tras H_{ux}\tras + \ub\tras H_{uu},
\\
 \label{LinddotHv} 
0 &=& {\rm  Lin}\ \ddot H_v
=-\frac{{\rm d}^2}{{\rm d}t^2}(\pb F_v + \xb\tras H_{vx}\tras),\quad {\rm  a.e.}\ {\rm  on}\ [0,T],\\ 
\label{LinHvT}
0 &=& ({\rm  Lin}\ H_{v})_T = \pb_T F_{v,T} + \xb_T\tras H_{vx,T}\tras,\\
\label{LindotHv0}
0 &=& ({\rm  Lin}\ \dot H_v)_0 =-\left.\ddt\right|_{t=0}(\pb F_v + \xb\tras H_{vx}\tras),
\end{eqnarray}
where we used equation \eqref{LinHv} and the commutation property of Lemma \ref{lindiff}.
Observe that \eqref{LinddotHv} -\eqref{LindotHv0} and Lemma \ref{lindiff} yield
\be
\label{LinHv0}
0={\rm  Lin}\ H_v=  \pb F_v + \xb\tras H_{vx}\tras,\quad {\rm a.e.}\ {\rm  on}\ [0,T].
\ee

\noindent\textbf{Notation:} denote by (LS) the set of equations consisting of \eqref{lineareq}, \eqref{linearconseq}, \eqref{linearcostate}-\eqref{LindotHv0}.

\if{
Once the linearized system (LS)  is computed, we can write the derivative of $\shoot$ in a direction  $\bar\nu:=\begin{pmatrix} \xb_0, \pb_0, \bar\beta \end{pmatrix}$ as follows,
\be
\shoot'
(\hat\nu)\bar\nu
=
\begin{pmatrix}
\vspace{2pt} D\eta(\xh_0,\xh_T)(\xb_0,\xb_T)
\\
\vspace{2pt}\pb_0+\left[\xb_0\tras D^2_{x_0^2}\ell
+\xb_T\tras D^2_{x_0x_T}\ell 
+\sum_{j=1}^{d_{\eta}}\bar{\beta}_jD_{x_0}\eta_j\right]_{(\xh_0,\xh_T)}\\
\vspace{2pt}\pb_T-\left[\xb_T\tras D^2_{x_T^2}\ell+\xb_0\tras D^2_{x_0x_T}\ell 
+\sum_{j=1}^{d_{\eta}}\bar{\beta}_jD_{x_T}\eta_j\right]_{(\xh_0,\xh_T)}\\
\vspace{2pt}\pb_TB_T+\xb_T\tras C_T\\
\left.\ddt\right|_{t=0}(\pb B + \xb\tras C\tras)
\end{pmatrix},
\ee
where  $(\xb,\ub,\vb,\pb)$ is the solution of \eqref{lineareq},\eqref{linearcostate},\eqref{LinHu0},\eqref{LinddotHv} associated with the initial condition $(\xb_0,\pb_0)$ and the multiplier $\bar\beta.$
}\fi

\begin{proposition}\label{LSunique}
The differential $\shoot'(\hat\nu)$  is one-to-one if the unique solution of  \eqref{lineareq}, \eqref{linearcostate}, \eqref{LinHu0}, \eqref{LinddotHv} with the initial conditions $(\xb_0,\pb_0)=0$ and with $\bar\beta=0,$ is $(\xb,\ub,\vb,\pb)=0.$ 
\end{proposition}


\subsection{Auxiliary linear-quadratic problem}\label{SecLQ}
Now we introduce the following linear-quadratic control problem in the variables $(\xib,\ub,\yb,\hb).$
Denote by (LQ) the problem given by
\begin{eqnarray}
&&\label{tcost} \Omega_{\P_2}(\xib,\ub,\yb,\hb)\rightarrow \min,\\
&&\label{tlineareq}\text{\eqref{xieq},\eqref{tlinearconseq}},\\
&&\label{heq}\dot{\hb}=0.
\end{eqnarray}
Here $\ub$ and $\yb$ are the control variables, $\xib$ and $\hb$ are the state variables, and $\Omega_{\P_2}$ is the quadratic mapping defined in \eqref{OmegaP2} associated with $\hat\lambda.$

\if{ 
Note that if condition \eqref{unifpos} holds then (LQ) has a unique optimal solution $(\xi,y,h)=0.$
Furthermore,  \eqref{unifpos} yields the strong convexity of the pre-Hamiltonian of (LQ) and hence its unique optimal solution is characterized by its first order optimality system. 
Here we present a one-to-one linear mapping that transforms each solution of (LS)  (introduced in paragraph \ref{SecLS} above) into a solution of this new optimality system. Theorem \ref{wp} will follow.
}\fi

Let $\bar\chi$ and $\bar\chi_h$ be the costate variables corresponding to $\xib$ and $\hb,$ respectively.
Note that the qualification hypothesis in Assumption \ref{cq}
implies that also the endpoint constraints \eqref{tlinearconseq} are also qualified. 
Hence any weak solution $(\xib,\ub,\yb,\hb)$ of (LQ) has a unique associated multiplier $\lambda^{LQ}:=(\bar\chi,\bar\chi_h, \beta^{LQ})$ solution of the system that we describe next (the multiplier associated with the cost function is fixed to 1) .
The pre-Hamiltonian for (LQ) is
\be\label{calH}
\begin{split}
&\mathcal{H}\lamLQ(\xib,\ub,\yb)
:= 
\bar\chi(F_x\xib + F_u\ub + B\yb ) \\
&+
(\half\xib\tras H_{xx} \xib + \ub\tras H_{ux}\xib + \yb\tras M\xib 
+ \half\ub\tras H_{uu} \ub
+ \yb\tras E \ub + \half\yb\tras R\yb ),
\end{split}
\ee
and the endpoint Lagrangian is given by
\be\label{ellLQ}
\ell^{LQ}\lamLQ(\xib_0,\xib_T,\hb_T):=
g(\xib_0,\xib_T,\hb_T)
+\sum_{j=1}^{d_{\eta}}\beta_j^{LQ} D\eta_j(\xib_0,\xib_T+F_{v,T}\hb_T).
\ee
The costate equation for $\bar\chi$ is
\be
\label{chi}
-\dot{\bar\chi}_t  
=
D_{\xib} \mathcal{H}\lamLQ 
= 
\bar\chi F_x + \xib\tras H_{xx} + \ub\tras H_{ux} + \yb^\top M,
\ee
with the boundary conditions
\be 
\label{chi0}
\begin{split}
\bar\chi_0=&
-D_{\xib_0}\ell^{LQ}\lamLQ\\
=&
- \xib_0\tras D_{x_0^2}^2\ell-(\xib_T+F_{v,T}\hb)\tras D_{x_0x_T}^2\ell 
-\sum_{j=1}^{d_{\eta}}\beta_j^{LQ} D_{x_0}\eta_j,
\end{split}
\ee
and
\be
\label{chiT}
\begin{split}
\bar\chi_T&=
D_{\xi_T}\ell^{LQ}\lamLQ\\
&=
 \xib_0\tras D_{x_0x_T}^2\ell
+(\xib_T+F_{v,T}\hb)\tras D^2_{x_T^2}\ell 
+\hb\tras H_{vx,T}
+\sum_{j=1}^{d_{\eta}}\beta_j^{LQ} D_{x_T}\eta_j.
\end{split}
\ee
For the costate variable $\bar\chi_h$ we get the equation and endpoint conditions
\begin{align} 
\dot{\bar\chi}_{h}&=0,\\
\bar\chi_{h,0}&=0,\\
\label{chihT0}\bar\chi_{h,T}
&=D_{\hb}\ell^{LQ}\lamLQ.
\end{align}
Hence, $\bar\chi_h\equiv 0$ and thus \eqref{chihT0} yields
\be\label{chihT}
\begin{split}
0= &\, \xib_0\tras D^2_{x_0x_T}\ell\,F_{v,T} +
(\xib_T+F_{v,T}\hb)\tras D^2_{x_T^2}\ell\, F_{v,T}+\xib_T\tras H_{vx,T}\tras +h\tras S_T  \\
& + \sum_{j=1}^{d_{\eta}}\beta_j^{LQ} D_{x_T}\eta_jF_{v,T}.
\end{split}
\ee
The {stationarity with respect to the control $(\ub,\yb)$} implies
\begin{gather}
\label{calHu}
0=\H_{\ub} = \bar\chi F_u + \xi\tras H_{ux}\tras + \ub\tras H_{uu} + \yb\tras E,
\\
\label{Hy}
0=\H_{\yb}=\bar\chi B+\xib\tras M\tras + \ub\tras E\tras + \yb \tras R.
\end{gather}
\noindent\textbf{Notation:} Denote by (LQS) the set of equations consisting of \eqref{tlineareq}-\eqref{heq},
\eqref{chi}-\eqref{chiT},\eqref{chihT}-\eqref{Hy}.

\bs

Note that  if the uniform positivity \eqref{unifpos} holds, then (LQ) has a unique optimal solution $(\xib,\ub,\yb,\hb)=0.$ Besides, in view of Corollary \ref{CoroSC}, the strengthened Legendre-Clebsch condition holds for (LQ) at $(\xib,\ub,\yb,\hb)=0.$ 
Hence, the unique local optimal solution of (LQ) is characterized by its first order optimality system (LQS). 
This leads to the following result.

\begin{proposition}\label{LQSunique}
If the uniform positivity in \eqref{unifpos} holds, the system (LQS) has a unique solution $(\xib,\ub,\yb,\hb)=0.$
\end{proposition}


\subsection{The transformation}\label{SecTransf}

Given $(\xb,\ub,\vb,\pb,\bar\beta)\in \W \times W^{1,\infty} \times \cR^{d_{\eta},*},$ define
\be\label{transf}
\yb_t:=\int_0^t \vb_sds,\ \xib:=\xb-F_v\yb,\ \bar\chi:=\pb+\yb\tras H_{vx},\ 
\bar\chi_h:=0,\ \hb:=\yb_T,\  
\beta_j^{LQ}:=\bar\beta_j.
\ee

\begin{lemma}
\label{lemmatransf}
If $\wh$ is a weak solution of \eqref{cost}-\eqref{finaleq}, the one-to-one linear mapping 
$
(\xb,\ub,\vb,\pb,\bar\beta)  \mapsto (\xib,\ub,\yb,\hb,\bar\chi,\bar\chi_h,\beta^{LQ})
$
defined by \eqref{transf} converts each solution of (LS) into a solution of (LQS).
\end{lemma}

\begin{proof} 
This proof is quite technical since it consists merely of the transformation of  the equations defining (LS).
Let $(\xb,\ub,\vb,\pb,\bar\beta)$ be a solution of (LS), and set
$(\xib,\ub,\yb,\hb,\bar\chi,\beta^{LQ})$  by \eqref{transf}. We want to prove that $(\xib,\ub,\yb,\hb,\bar\chi,\beta^{LQ})$ satisfies \eqref{tlineareq}-\eqref{heq}, \eqref{chi}-\eqref{chiT},\eqref{chihT}-\eqref{Hy}.

\noindent\textit{Part I.} 
Regarding the state equations, observe that \eqref{xieq} follows by differentiating the expression of $\xib$ in \eqref{transf}, 
and \eqref{tlinearconseq} follows from  \eqref{linearconseq}. The equation for $\hb$ in \eqref{heq} is an immediate consequence of its definition.

\noindent\textit{Part II.}
We shall prove that $(\xib,\ub,\yb,\hb,\bar\chi,\beta^{LQ})$ verifies the costate equations given by \eqref{chi}-\eqref{chiT} and \eqref{chihT}.
Differentiate $\bar\chi$ in \eqref{transf}, use equations \eqref{linearcostate} and \eqref{transf}, recall definition of $M$ in \eqref{M} and obtain,
\be
\begin{split}
-\dot{\bar\chi}
=& -\dot \pb-\vb\tras H_{vx}-\yb\tras \dot{H}_{vx}\\ 
=& \pb F_x + \xb\tras H_{xx} +\ub\tras H_{ux} -\yb\tras \dot{H}_{vx}\\
=& \bar\chi F_x+\xib\tras H_{xx} +\ub\tras H_{ux} + \yb\tras (-H_{vx}F_x +F_v\tras H_{xx}-\dot{H}_{vx})\\
=& \bar\chi F_x + \xib\tras H_{xx}+\ub\tras H_{ux}+\yb\tras M.
\end{split}
\ee
Hence \eqref{chi} holds.
Equations \eqref{chi0}-\eqref{chiT} follow from \eqref{condq0}-\eqref{condqT}.
Combine \eqref{condqT} and \eqref{LinHvT} to get
\be
\label{tlinPhiT}
\begin{split}
 0
=&\, \pb_T F_{v,T} + \xb_T\tras H_{vx,T}\tras\\
=& \Big[\pb_T\tras D^2_{x_T^2}\ell +\xb_0\tras D^2_{x_0x_T}\ell  
+\sum_{j=1}^{d_{\eta}}\bar{\beta}_jD_{x_T}\eta_j\Big]F_{v,T}+ \xb_T\tras H_{vx,T}\tras.
\end{split}
\ee
By performing transformation \eqref{transf} in the previous equation and recalling that $S_T=F_{v,T}\tras H_{vz,T}\tras$ (in view of Corollary \ref{CoroCBsym}) one obtains \eqref{chihT}.

\noindent\textit{Part III.} Let us show that the stationarity with respect to $\yb$ in \eqref{calHu} is verified. The transformation in  \eqref{transf} together with  equation \eqref{LinHu0} imply
\be 
\begin{split}
0 = &\,(\bar\chi - \yb\tras H_{vx})F_u + (\xib + F_v\yb)\tras H_{ux}\tras + \ub\tras H_{uu}
\\
=&\,\bar\chi F_u + \xib\tras H_{ux}\tras + \ub\tras H_{uu} + \yb\tras(F_v\tras H_{ux}\tras -H_{vx}F_u).
\end{split}
\ee
Calling back definition of $E$ in \eqref{M}, stationarity condition \eqref{calHu} follows.

\noindent\textit{Part IV.} Finally, we shall prove that \eqref{Hy} holds. Perform the transformation \eqref{transf} in equation \eqref{LinHv0} to obtain
\benl 
0=(\bar\chi - \yb\tras H_{vx})F_v +  (\xib + F_v\yb)\tras H_{vx}\tras = \bar\chi F_v + \xib\tras H_{vx}\tras,
\eenl
since Corollary \ref{CoroCBsym} holds when the multiplier in unique.
Differentiating previous expression we obtain
\be
\label{Hy0}
\begin{split}
0= &-(\bar\chi F_x + \xib\tras H_{xx} + \ub\tras H_{ux} + \yb\tras M)\,F_v  +\bar\chi \dot{F}_v \\
&+ (F_x\xib+F_u\ub+B\yb)\tras H_{vx}\tras +\xib\tras \dot{H}_{vx}\tras.
\end{split}
\ee
Recall the definitions of $B$ in \eqref{B1} and of $E$ in \eqref{M}, of $R$ in \eqref{R1}, and use them in \eqref{Hy0} to get \eqref{Hy}.

Parts I to IV show that $(\xib,\ub,\yb,\hb,\bar\chi,\beta^{LQ})$ is a solution of (LQS), and hence the result follows.
\end{proof}


We shall now go back to the convergence Theorem \ref{wp}.

\begin{proof}{[of Theorem \ref{wp}]}
Let $(\xb,\ub,\vb,\pb,\bar\beta)$ be a solution of (LS), and let 
\\
$(\xib,\ub,\yb,\hb,\bar\chi,\bar\chi_h,\beta^{LQ})$ be defined by the transformation in \eqref{transf}. Hence we know by Lemma \ref{lemmatransf} that $(\xib,\ub,\yb,\hb,\bar\chi,\bar\chi_h,\beta^{LQ})$
is solution of (LQS). 
As it has been already shown in Proposition \ref{LQSunique}, condition \eqref{unifpos} implies that the unique solution of (LQS) is 0.
Hence $(\xib,\ub,\yb,\hb,\bar\chi,\bar\chi_h,\beta^{LQ})=0$ and thus $(\xb,\ub,\vb,\pb,\bar\beta)=0.$ Conclude that the unique solution of (LS) is 0. This yields the injectivity of $\shoot'$ at $\hat\nu,$ and hence the convergence result follows.
\end{proof}

\begin{remark}
[The shooting algorithm for the control constrained case] We claim that the formulation of the shooting algorithm above and the proof of its local convergence can be also done for problems where the controls are subject to bounds of the type
\be
0\leq u_t \leq 1,\quad 0\leq v_t \leq 1,\quad \text{a.e.}\ \text{on}\ [0,1],
\ee
and under some geometric hypotheses on the structure of the optimal control.
This extension should follow the procedure in Section 8 of \cite{ABM11}. 
\end{remark}

\if{
\section{Examples}

Lawden's spiral and example in \cite{DmiShi10}, some example from geometry.
}\fi

\section{Conclusion}

We studied optimal control problems in the Mayer form with systems that are affine in some components of the control variable. 
A set of `no gap' necessary and sufficient second order  optimality conditions is provided. These conditions apply to a weak minimum and do not assume the uniqueness of multipliers.
For qualified solutions, we proposed a shooting algorithm and proved that its local convergence is guaranteed by the sufficient condition above-mentioned.

There are several issues in this direction of investigation that remain open. For instance, one can think of the study of other type of minimum, like Pontryagin or strong. Other possible task is the optimality of bang-singular solutions, that had not yet been deeply looked into but show to be useful in practice. Therefore, the results presented here can be pursued by many interesting extensions.

\if{
\begin{proof}
Let us compute $\ddt\,{\rm  Lin}\, \G:$
\be\label{ddtLin}
\ba{rl}
&\ddt\,\left[{\rm  Lin}\, \G\mid_{(\zeta^0_t,\alpha^0_t)}(\bar{\zeta}_t,\bar{\alpha}_t)\right]\\
&=[\G_{\zeta\zeta}(\zeta^0_t,\alpha^0_t)\dot\zeta^0_t+
\G_{\zeta\alpha}(\zeta^0_t,\alpha^0_t)\dot\alpha^0_t]\bar\zeta_t
+\G_{\zeta}(\zeta^0_t,\alpha^0_t)\dot{\bar\zeta}_t\\
&+[\G_{\alpha\zeta}(\zeta^0_t,\alpha^0_t)\dot\zeta^0_t+
\G_{\alpha\alpha}(\zeta^0_t,\alpha^0_t)\dot\alpha^0_t]\bar\alpha_t
+\G_{\alpha}(\zeta^0_t,\alpha^0_t)\dot{\bar\alpha}_t.
\ea
\ee
For ${\rm  Lin}\,\ddt \G$ we get
\be\label{Linddt}
\ba{rl}
&{\rm  Lin}\mid_{(\zeta^0_t,\alpha^0_t)}\,\ddt \G(\zeta_t,\alpha_t)\\
&={\rm  Lin}\mid_{(\zeta^0_t,\alpha^0_t)}
[\G_{\zeta}(\zeta_t,\alpha_t)\dot\zeta_t+
\G_{\alpha}(\zeta_t,\alpha_t)\dot\alpha_t]\\
&=[\G_{\zeta\zeta}(\zeta^0_t,\alpha^0_t)\bar\zeta_t
+\G_{\zeta\alpha}(\zeta^0_t,\alpha^0_t)\bar\alpha_t]\dot{\zeta}^0_t
+\G_{\zeta}(\zeta^0_t,\alpha^0_t)\dot{\bar\zeta}_t\\
&+[\G_{\alpha\zeta}(\zeta^0_t,\alpha^0_t)\bar\zeta_t
+\G_{\alpha\alpha}(\zeta^0_t,\alpha^0_t)\bar\alpha_t]
\dot\alpha^0_t
+\G_{\alpha}(\zeta^0_t,\alpha^0_t)\dot{\bar\alpha}_t.
\ea
\ee
By \eqref{ddtLin} and \eqref{Linddt} we get \eqref{lemmaLineq}.
\end{proof}
}\fi


\end{document}

\end{document}
